\documentclass[11pt]{amsart}
\usepackage[leqno]{amsmath}
\usepackage{epsfig}
\usepackage{epsfig}
\usepackage{color}
\usepackage{amsmath}
\usepackage{amssymb}
\newtheorem{lemma}{Lemma}
\newtheorem{proposition}{Proposition}
\newtheorem{theorem}{Theorem}
\newtheorem{corollary}{Corollary}

\theoremstyle{definition}

\usepackage{color}
\newcommand{\commentout}[1]{}
\usepackage{graphicx}
\newcommand{\ignore}[1]{}

\begin{document}

\thispagestyle{empty}

\centerline{\Large\bf Constant approximation algorithms for
embedding } \centerline{\Large\bf graph metrics into trees and
outerplanar graphs\footnote{\noindent An
extended abstract of this paper will appear in the proceedings of APPROX-RANDOM 2010}} 

\vspace{6mm}

\centerline{{\sc V. Chepoi$^{\small 1}$, F. F. Dragan$^{\small 2},$  I. Newman$^{\small 3}$, Y. Rabinovich$^{\small 3},$ and   Y. Vax\`es}$^{\small 1}$}

\vspace{3mm}
\medskip
\centerline{$^{1}$Laboratoire d'Informatique Fondamentale,}
\centerline{Universit\'e d'Aix-Marseille,}
\centerline{Facult\'e des Sciences de Luminy,} \centerline{F-13288
Marseille Cedex 9, France} \centerline{\{chepoi, vaxes\}@lif.univ-mrs.fr}

\medskip
\centerline{$^{2}$Computer Science Department,}
\centerline{Kent State University,}
\centerline{Kent, OH 44242, USA}
\centerline{dragan@cs.kent.edu}

\medskip
\centerline{$^{3}$Department of Computer Science,}
\centerline{University of Haifa,}
\centerline{Mount Carmel, Haifa 31905, Israel}
\centerline{\{ilan,yuri\}@cs.haifa.ac.il}

\bigskip

\vspace{7mm}
\begin{footnotesize} \noindent {\bf Abstract.} In this paper, we present a simple factor 6 algorithm for approximating the optimal multiplicative distortion of embedding
a graph metric into a tree metric  (thus improving and simplifying the factor 100 and 27 algorithms of B\v{a}doiu, Indyk, and Sidiropoulos (2007) and
B\v{a}doiu, Demaine, Hajiaghayi, Sidiropoulos, and Zadimoghaddam (2008)). We also present  a constant factor algorithm for approximating the
optimal distortion of embedding a graph metric into an outerplanar metric. For this, we introduce a general notion
of metric relaxed minor and show that if $G$ contains an $\alpha$-metric relaxed $H$-minor, then the distortion
of any embedding of $G$ into any metric induced by a $H$-minor free graph
is $\geq \alpha$. Then, for $H=K_{2,3}$, we present an algorithm  which either finds an $\alpha$-relaxed minor, or produces an
$O(\alpha)$-embedding into an outerplanar metric.
\end{footnotesize}

\section{Introduction}

\subsection{Avant-propos}

The structure of the shortest-path metrics of special classes of
graphs, in particular, graph families defined by a set of forbidden minors
(e.g., line metrics, tree metrics, planar metrics)
is one of the main areas in the theory of metric spaces.
From the algorithmic point of view, such metrics typically have more
structure than general metrics, and this structure can often be exploited
algorithmically. Thus, if the input metric can be well approximated
by a special metric, this usually leads to an algorithmic advantage;
see, e.g.,~\cite{InMa} for a survey of embeddings and their algorithmic
applications. One way of understanding this structure is to study
the low distortion embeddings from one metric class to another. To do
this successfully, one needs to develop tools allowing a decomposition of the host space
consistent with the embedded space. If this is impossible, one
usually learns much about the limitations of the host space and the richness
of the embedded space. In this paper we pursue this direction and study the embeddings
into tree metrics and the metrics of $K_{2,3}$-minor free graphs (the outerplanar
metrics essentially,   because each 2-connected component of a $K_{2,3}$-minor free graph
is either outerplanar or a $K_4$).

The study of tree metrics can be traced back to the beginning of the
20th century, when it was first realized that weighted trees can in some
cases serve as an (approximate) model for the description of evolving
systems.  More recently, as indicated in~\cite{ShTa}, it was observed
that certain Internet originated metrics display tree-like properties.  It is well known~\cite{SeSt} that tree
metrics have a simple structure: $d$ is a tree metric
if and only if all submetrics of $d$ of size 4 are such.  Moreover, the
underlying tree is unique, easily reconstructible, and has rigid local structure
corresponding to the local structure of $d$. But what about the structure of {\em
approximately} tree metrics? We have only partial answers
for this question, and yet what we already know seems to indicate that
a rich theory might well be hiding there. The strongest results were obtained, so far, for the {\em additive}
distortion. A research on
the algorithmic aspects of finding a tree metric of least additive
distortion  has culminated in the paper \cite{AgBaFaNaPa}  (see also
\cite{ChFi}), where a 6-approximation algorithm was established (in
the notation of \cite{AgBaFaNaPa}, it is a $3$-approximation algorithm,  however, in our more restrictive definition,
requiring that the metric is dominated by the approximating one,
it is a $6$-approximation), together with a (rather close)
hardness result. Relaxing the local condition on $d$ by allowing
its size-4 submetrics to be $\delta$-close to a tree metric, one gets
precisely Gromov's $\delta$-hyperbolic geometry. For study of
algorithmic and other aspects of such geometries, see e.g.~\cite{ChDrEs++,KrLe}.

The situation with the {\em multiplicative} distortion is less satisfactory.
The best result for embedding general metrics into tree metrics is obtained in~\cite{BaInSi}:
the approximation factor is exponential in $\sqrt{\log \Delta }/\log\log n$,
where $\Delta$ is the aspect ratio. Judging from the parallel results
of~\cite{BaChInSi} for embedding into line metrics, it is conceivable that any constant factor
approximation for optimal embedding general metrics into tree metrics is NP-hard.  For some small
constant $\gamma$, the hardness result of \cite{AgBaFaNaPa}  implies that it is NP-hard
to approximate the multiplicative distortion better than $\gamma$ even for metrics that
come from unit-weighted graphs. For a special interesting case of shortest path metrics of {\it unit-weighted}
graphs, \cite{BaInSi} gets a large (around 100) constant approximation factor
(which was improved in \cite{BaDeHaSiZa}
to a factor 27). The proof introduces a certain metric-topological obstacle for getting
embeddings of distortion better than $\alpha$, and then algorithmically
either produces an $O(\alpha)$-embedding, or an $\alpha$-obstacle.
Let us mention that such an obstacle was used also in~\cite{EP04}, and,
essentially, in~\cite{RR98}.

\subsection{Our results}
In this paper, we study the embeddings of {\em unweighted} (i.e., unit-weighted)
{\em graph} metrics into tree metrics and outerplanar metrics.
Using  a decomposition procedure developed earlier in~\cite{BrChDr,ChDr}, we
simplify and improve the construction of~\cite{BaInSi} for embedding into tree metrics.
The improved constant is 6. We also introduce the notions
of relaxed and metric relaxed minors and show that
if $G$ contains an $\alpha$-metric relaxed $H$-minor, then the distortion
of any embedding of the metric of $G$ into any metric induced by a $H$-minor free graph
is at least $\alpha$. This generalizes the obstacle of~\cite{BaInSi}.
Using this newly defined $H$-obstacle, we are able to show that it is an essential
obstacle not only for trees, but also for graphs without  $H=K_{2,3}$ minors as well.
We further develop an efficient algorithm which either embeds the input graph $G$ into
an outerplanar metric with distortion $O(\alpha)$, or finds an $\alpha$-metric
relaxed $K_{2,3}$-minor in $G$. This is a first result of this kind for any
$H$ different from a $C_4$ (which is the
corresponding $\alpha$-metric relaxed minor corresponding to the four-point
condition used for embedding into tree-metrics). It is our feeling that this obstacle may prove essential
for other forbidden $H$'s, notably $K_{2,r}$, hopefully series-parallel graphs,
and beyond.

\subsection{Preliminaries}
A metric space $(X,d)$ is {\it isometrically embeddable} into a host
metric space $(Y,d')$ if there exists a map $\varphi: X\mapsto Y$
such that $d'(\varphi(x),\varphi(y))=d(x,y)$ for all $x,y\in X.$ In
this case we say that $X$  is a subspace of $Y.$ More generally,
$\varphi: X\mapsto Y$ is an {\it embedding with (multiplicative)
distortion} $\lambda\ge 1$ if  $d(x,y)\le
d'(\varphi(x),\varphi(y))\le \lambda\cdot d(x,y)$ for all $x,y\in X$
(note that embedding here is non-contracting; this could be relaxed
to contracting embeddings as well, but will not be important in what
follows). Given a metric space $(X,d)$ and a class ${\mathcal M}$ of
host metric spaces, we denote by $\lambda^*:=\lambda^*(X, {\mathcal
M})$ the minimum distortion of an embedding of $(X,d)$ into a member
of $\mathcal M$.  Analogously, $\varphi: X\mapsto Y$ is an {\it
embedding with additive distortion} $\lambda\ge 0$ if  $d(x,y)\le
d'(\varphi(x),\varphi(y))\le d(x,y)+\lambda$ for all $x,y\in X.$ In
a similar way, we can define the minimum additive distortion for
embedding of a metric space $(X,d)$ into a  class ${\mathcal M}$ of
host metric spaces.  In this paper, we consider finite connected
unweighted graphs as input metric spaces and {\it tree metrics}
(trees) or {\it outerplanar metrics}  (and they relatives)  as the
class of host metric spaces. If not specified, all our results
concern embeddings with multiplicative distortion. For a connected
unweighted graph $G=(V,E),$ we denote by $d_G(u,v)$ the
shortest-path distance between two vertices $u$ and $v$ of $G.$  A
finite metric  space $(X,d)$ is called a {\it tree metric} if it
isometrically embeds into a tree, i.e., there exists a weighted
tree $T=(X',E')$ such that $X\subseteq X'$ and $d(u,v)=d_T(u,v)$ for
any two points $u,v\in X,$ where $d_T(u,v)$ is the length of the
unique path connecting $u$ and $v$ in $T.$ Analogously, an {\it
outerplanar metric} is a metric space isometrically embeddable into
an outerplanar weighted graph. We denote by $\mathcal T$ the class
of all tree metric spaces and by $\mathcal O$ the class of
outerplanar metric spaces. Note that $\mathcal T$ is a proper
subclass of $\mathcal O$.

\section{Preliminary results}
In this section, we establish some properties of layering  partitions and of embeddings with distortion $\lambda$ of graph metrics into weighted graphs.

\subsection{Layering partitions}

We now briefly describe the layering partitions and establish some of their properties. The layering partitions have been
introduced in the papers \cite{BrChDr,ChDr} and recently
used in a slightly more general forms in both approximation algorithms  for embedding graph metric into trees
\cite{BaDeHaSiZa,BaInSi} as well as in some other similar contexts  \cite{ChDrEs++,DoDrGaYa,DoGa}.

Let $G=(V,E)$ be an unweighted
connected graph with a distinguished vertex $s$ and let $r:=\max\{d_G(s,x): x\in V\}$. A
{\it layering} of $G$ with respect to $s$ is the
decomposition of $V$ into the {\it spheres} $L^i=\{ u\in
V: d(s,u)=i\},$ $i=0,1,2,\ldots,r.$ A {\it layering partition}
${\mathcal LP}(s)=\{L^i_1,\ldots,L^i_{p_i} : i=0,1,2,\ldots, r\}$ of $G$ is
a partition of each $L^i$ into {\it clusters}
$L^i_1,\ldots,L^i_{p_i}$ such that two vertices $u,v\in L^i$ belong
to the same  cluster $L^i_j$ if and only if they can be connected by
a path outside the ball $B_{i-1}(s)$ of radius $i-1$ centered at
$s$. It was shown in \cite{ChDr} that
for a given unweighted graph $G$ such a layering partition can be
found in linear time. Let $\Gamma$ be a graph whose vertex set is the set of all clusters
$L_j^i$ in a layering partition ${\mathcal LP}$ of a graph $G.$ Two
vertices $C=L_j^i$ and $C'=L_{j'}^{i'}$ are adjacent in $\Gamma$ if and
only if there exist $u\in L_j^i$ and $v\in L^{i'}_{j'}$ such that
$u$ and $v$ are adjacent in $G$ (see Fig. \ref{fig:partition}). It
is shown in \cite{ChDr} that $\Gamma$ is a tree, called the {\it
layering tree} of $G$, and that $\Gamma$ is computable in linear
time in the size of $G$.  In what follows, we assume that
$\Gamma$ is rooted at cluster  $\{s\}$.

\medskip


\begin{figure}[tbh]
\vspace*{-1.2cm}
\begin{center}
\includegraphics[height=5.0cm]{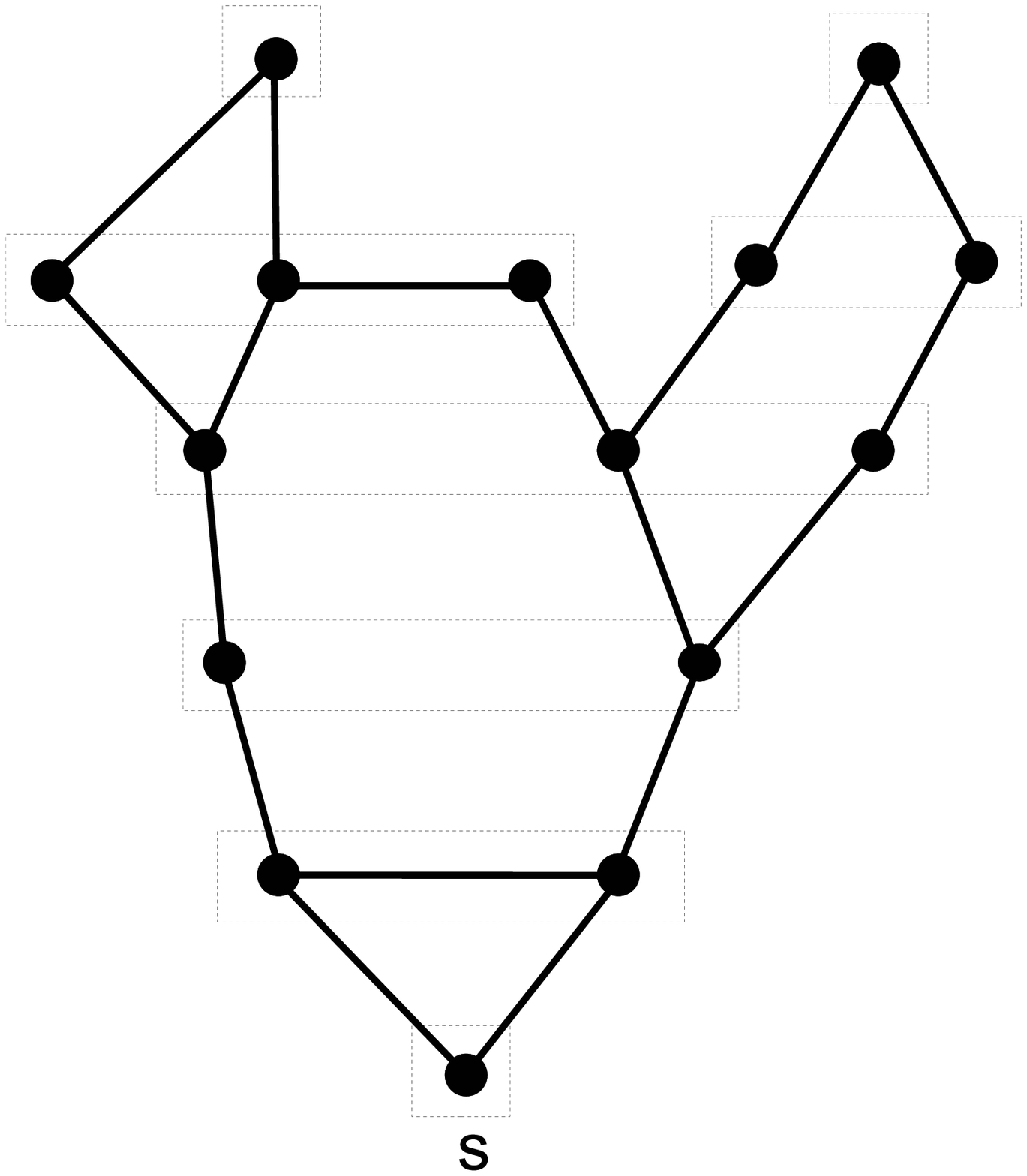}\hspace*{-0.1cm}
\includegraphics[height=5.0cm]{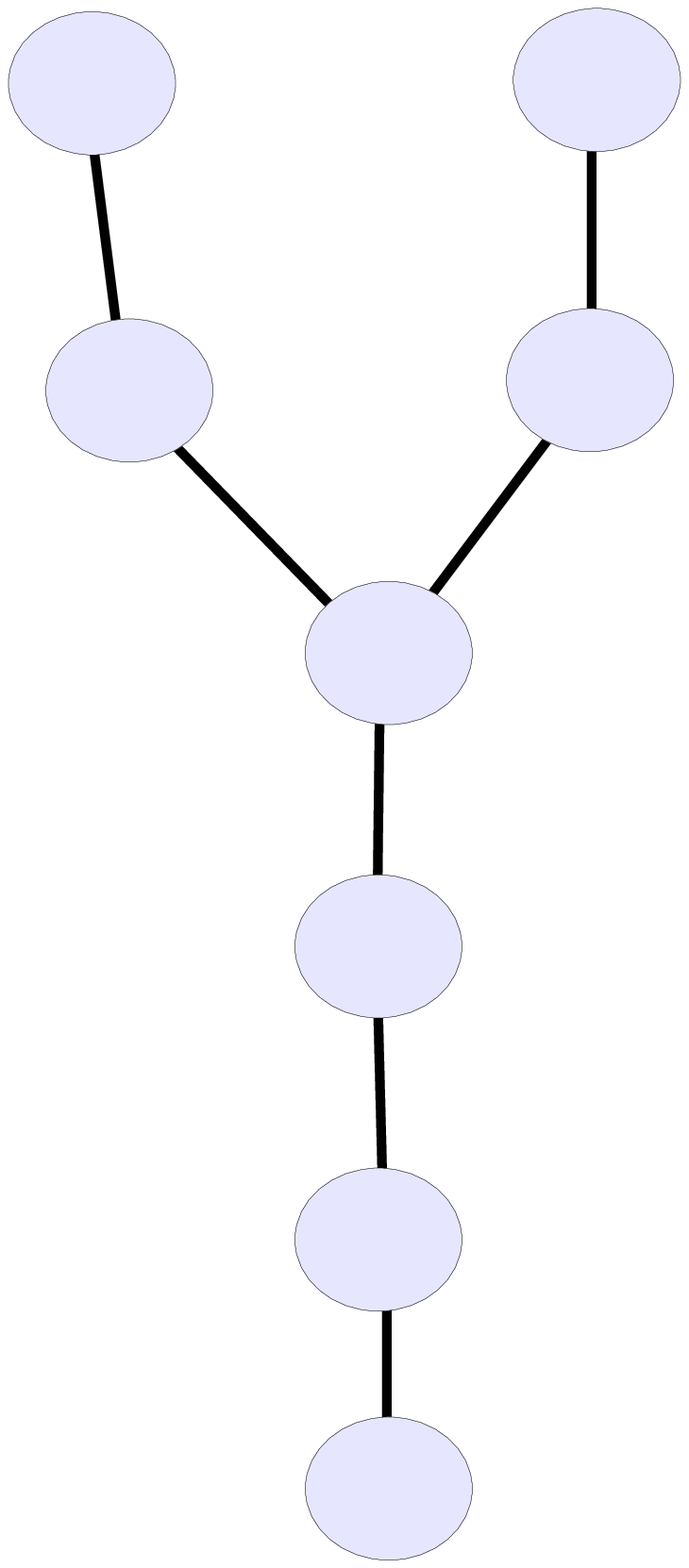}\hspace*{-1.4cm}
\includegraphics[height=5.5cm]{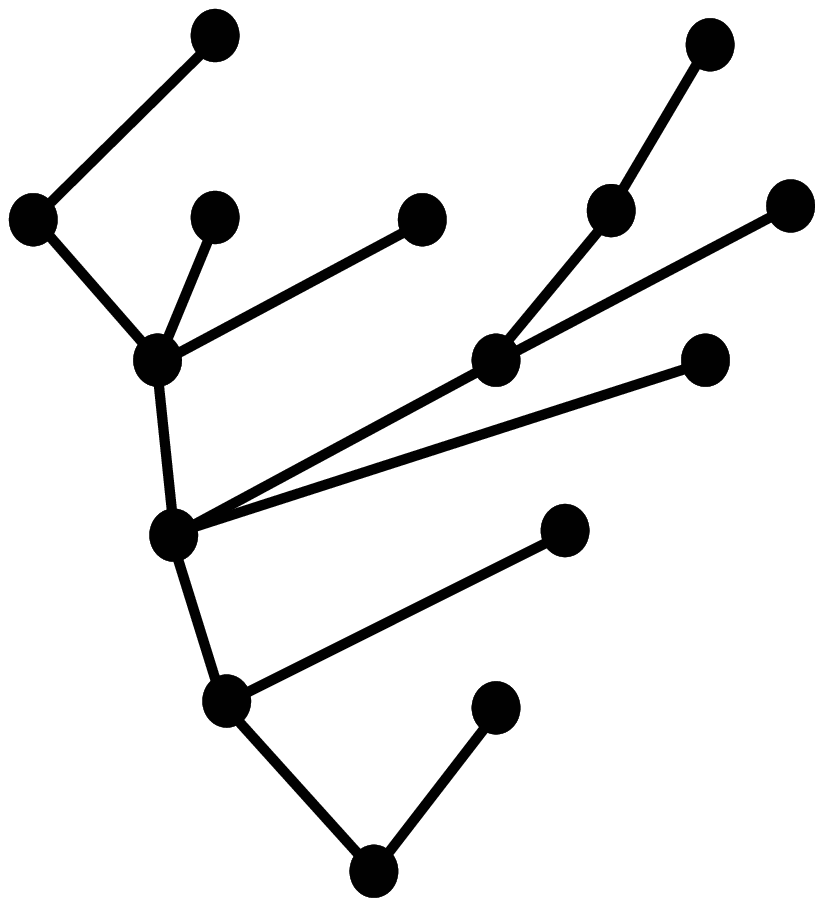}
\end{center}
\vspace*{-1cm} \caption{\label{fig:partition} A layering partition of $G$ and the trees $\Gamma$ and  $H$ associated with this layering
partition.}
\end{figure}

We can construct a new tree $H=(V,F)$ for a graph $G$ (closely reproducing the global structure
of the layering tree $\Gamma$) by identifying
for each cluster $C=L^i_j\in {\mathcal LP}$ an arbitrary vertex
$x_C\in L^{i-1}$ which has  a neighbor in $C=L^i_j$ and by making
$x_C$ adjacent in $H$ with all vertices $v\in C$ (see the rightmost
picture in Fig. \ref{fig:partition}). Vertex $x_C$ will be called
{\it support vertex} for cluster $C=L^i_j$. In what follows, we
assume that $H$ is rooted at vertex $s$.

Let $D$ be the largest diameter of a cluster in a layering
partition ${\mathcal LP}$ of $G$, i.e., $D:=\max_{C\in
{\mathcal LP}}\max_{v,u\in C}\{d_G(u,v)\}$. Then, the following
result (also implicitly  used in \cite{BrChDr,ChDr,ChDrEs++} in
particular cases) shows that the additive distortion of the
embedding of $G$ into $H$ is essentially $D$:

\begin{proposition}\label{lm:treeH} For any vertices $x,y$ of $G,$
   $d_H(x,y)-2 \leq d_G(x,y)\leq d_H(x,y)+D.$
\end{proposition}

\begin{proof}
\ignore{ 
Let $C_x$ and $C_y$ be the clusters containing
the vertices $x$ and $y$, respectively. Denote by $C$ the cluster which is the
nearest common ancestor of $C_x$ and $C_y$ in the layering
tree $\Gamma$ of $G$. Let $x',y'\in C$ be the ancestors of $x$ and $y$, respectively,
in the tree $H.$  Let
$a:=d_H(x,x')$ and $b:=d_H(y,y')$. By construction of $H$, we have
$d_G(x,x')=d_H(x,x')=a$, $d_G(y,y')=d_H(y,y')=b$, and $d_H(x,y)$ is
equal either to $a+b$ or to $a+b+2$. Therefore, by triangle
inequality, $d_G(x,y)\leq d_G(x,x')+ d_G(x',y')+d_G(y,y')\leq a+b+
D\leq d_H(x,y)+D$. By the choice of $C$, we obtain also that
$d_G(x,y)\geq d_G(x,x')+ d_G(y,y')=a+b\geq d_H(x,y)-2$.}
Let $C_x$ and $C_y$ be the clusters containing the vertices $x$ and
$y$, respectively. Denote by $C$ the cluster which is the nearest
common ancestor of $C_x$ and $C_y$ in the layering tree $\Gamma$ of
$G$. The assertion is trivial if $C_x=C_y=C$. For $C \neq C_x$, let
$x',y'\in C$ be the ancestors of $x$ and $y$, respectively, in a
BFS(G,s)-tree. Then $d_{\Gamma}(C_x,C)=d_G(x,x')$ and
$d_{\Gamma}(C_y,C) =d_G(y,y')$. By construction of $H$, $d_H(x,y)$
is equal either to $d_{\Gamma}(C_x,C)+d_{\Gamma}(C_y,C)$ or to
$d_{\Gamma}(C_x,C)+d_{\Gamma}(C_y,C)+2$. Thus, by the triangle
inequality,
$$ d_G(x,y) \leq d_G(x,x') + d_G(x',y') + d_G(y,y') \leq
d_{\Gamma}(C_x,C)+d_{\Gamma}(C_y,C)+D \leq d_H(x,y) + D.
$$
On the other hand, by definition of clusters, $d_G(x,y) \geq
d_G(x,x') + d_G(y,y') \geq d_H(x,y)-2.$
\end{proof}

Note that tree $H,$ like any BFS-tree, preserves graph distances between the root
$s$ and any other vertex of $G.$ We can locally modify $H$ by assigning uniform
weights to its edges or by adding Steiner points to obtain a number of other desired
properties (like, non-expansiveness, non-contractibility, etc.).
For example, assigning length $w:=D+1$ to each edge of $H,$ we will get
a uniformly weighted tree $H_w=(V,F,w)$ in which $G$ embeds with multiplicative
distortion essentially equal to $D+1:$

\begin{corollary} \label{col:treeHw}
For any vertices $u,v$ of $G,$ $d_G(u,v)\leq d_{H_w}(u,v)\leq  (D+1)(d_G(u,v)+2).$
\end{corollary}

By adding Steiner points and using edge lengths 0 and 1, the tree $H$
can be easily transformed into a tree $H'$ which has the same
additive distortion and satisfies the non-expansive property. For
this, for each cluster $C:=L^i_j$ we introduce a Steiner point
$p_C,$ and add an edge of length 0 between any vertex of $C$ and
$p_C$ and an edge of length 1 between $p_C$ and the support vertex
$x_C$ for $C$, defined above.

\begin{corollary} \label{col:treeH'} For any vertices $u,v$ of $G,$
$d_{H'}(u,v)\leq d_{G}(u,v)\leq  d_{H'}(u,v)+D$.
\end{corollary}

By replacing each edge in $H'$ with edge of length
$w:=\frac{D+1}{2}$, we obtain a tree $H'_w$ with the following property:

\begin{corollary}\label{cor:H'w} For any vertices $u,v$ of $G,$
$d_{G}(u,v)\leq d_{H'_w}(u,v)\leq  (D+1) (d_{G}(u,v)+1)$.
\end{corollary}

\subsection{Embeddings with distortion $\lambda$ of graph metrics}

We continue with two auxiliary standard results about embeddings.

\begin{lemma} \label{distortion_edge1} If $G=(V,E)$ is an unweighted graph, $G'=(V',E')$
is a weighted graph, and  $\varphi: V\mapsto V'$ is a mapping  such that $d_{G'}(\varphi(u),\varphi(v))\le \lambda$
for any edge $uv$ of $G,$ then   $d_{G'}(\varphi(x),\varphi(y))\le \lambda d_G(x,y)$ for any
pair of vertices $x,y$ of $G.$
\end{lemma}

\begin{proof} Consider a shortest path $P$ of $G$ between arbitrary vertices $x,y$ of $G.$
For each edge $uv$ of $P$, $\varphi(u)$ and $\varphi(v)$
are connected in $G'$ by a path $P_{uv}$ of length $\le \lambda.$
Hence, $\varphi(x)$ and $\varphi(y)$ can be connected in the
subgraph of $G'$ induced by $\cup\{ P_{uv}: uv \mbox{ is an edge of
} P\}$ by a path with total length  of edges at most $\lambda
d_G(x,y).$ Hence, $d_{G'}(\varphi(x),\varphi(y))\le \lambda
d_G(x,y).$
\end{proof}

\begin{lemma} \label{distortion_edge2} If $G=(V,E)$ is an unweighted graph, $G'=(V',E')$
is a weighted graph, and  $\varphi: V\mapsto V'$ is a mapping  such that
$d_{G'}(\varphi(u),\varphi(v))\ge d_G(u,v)$ for any edge $\varphi(u)\varphi(v)$ of $G',$
then $d_{G'}(\varphi(x),\varphi(y))\ge d_G(x,y)$ for any pair of vertices $x,y$ of $G.$
\end{lemma}

\begin{proof} We proceed by induction on the number of edges in a shortest path between
$\varphi(x)$ and $\varphi(y)$ in $G'.$ If $\varphi(x)$
and $\varphi(y)$ are adjacent in $G',$ then we are done by our condition. Otherwise, let
$\varphi(x')$ be the neighbor of $\varphi(x)$ in such a
shortest path. By induction hypothesis, $d_{G'}(\varphi(x'),\varphi(y))\ge d_G(x',y).$
Since $d_{G'}(\varphi(x),\varphi(x'))\ge d_G(x,x'),$ the
triangle inequality  yields
$d_{G'}(\varphi(x),\varphi(y))=d_{G'}(\varphi(x),\varphi(x'))+d_{G'}(\varphi(x'),\varphi(y))\ge d_G(x,x')+d_G(x',y)\ge d_G(x,y),$
and we are done.
\end{proof}

\section{Approximation algorithm for embedding graph metrics into trees}

We describe now a simple factor 6
algorithm for approximating the optimal distortion
$\lambda^*=\lambda^*(G,{\mathcal T})$ of embedding finite unweighted graphs
$G$ into trees. For this, we first investigate the properties of
layering partitions of graphs which $\lambda$-embed into trees,
i.e., for each such graph $G=(V,E)$ there exists a tree $T=(V',E')$
with $V\subseteq V'$ such that
\begin{equation}
   d_G(x,y)\leq d_T(x,y)\hspace*{2cm}\mbox{(non-contractibility)} \label{eq1}
\end{equation}

\noindent
and
\begin{equation}
    d_T(x,y)\leq \lambda\cdot d_G(x,y)\hspace*{2cm}\mbox{(bounded expansion)} \label{eq2}
\end{equation}
for every $x,y\in V$.  Denote by $P_T(x,y)$
the path connecting the vertices $x,y$ in $T.$ For  $x\in V'$ and  $A\subseteq V',$ we denote by $d_T(x,A)=\min \{ d_T(x,v): v\in A\}$ the distance
from $x$ to $A.$ First we show that the diameters of clusters in a layering partition of such a graph $G$ are at most $3\lambda,$ allowing already to build a
tree with distortion $8\lambda^*$. Refining this property of layering partitions, we  construct in $O(|V||E|)$ time a  tree into which $G$
embeds with distortion $\le 6\lambda^*$.

\begin{lemma}\label{lm:distance_to_path} If a graph $G$ $\lambda$-embeds into a tree, then for any $x,y\in V,$ any path $P_G(x,y)$ of $G$ between $x,y$
and any vertex $c\in P_T(x,y),$ we have $d_T(c,P_G(x,y))\le \lambda/2.$
\end{lemma}

\begin{proof} Removing $c$ from $T$, we separate $x$ from $y.$ Let $T_y$ be the subtree of $T\setminus\{ c\}$ containing $y.$
Since $x\notin T_y,$  we can find an edge $ab$ of $P_G(x,y)$ with
$a\in T_y$ and $b\notin T_y.$ Therefore, the path $P_T(a,b)$ must go
via $c$. If $d_T(c,a)>\lambda/2$ and $d_T(c,b)>\lambda/2,$ then
$d_T(a,b)=d_T(a,c)+d_T(c,b)>\lambda$ and since $d_G(a,b)=1,$ we
obtain a contradiction with the assumption that the embedding of $G$
in $T$ has distortion $\lambda$ (condition (2)). Hence
$d_T(c,P_G(x,y))\le\min\{ d_T(c,a),d_T(c,b)\}\le \lambda/2,$
concluding the proof.
\end{proof}

\begin{lemma}\label{lm:diam} If a graph $G$ $\lambda$-embeds into a
  tree $T$, then the diameter in $G$ of any cluster $C$ of a layering partition of $G$
is at most $3\lambda,$ i.e., $d_G(x,y)\le 3\lambda$ for any two
vertices $x,y\in C.$ In particular, $\lambda^*(G,{\mathcal T})\ge
D/3,$ where $D$ is the maximal diameter of a cluster of a layering
partition of $G.$
\end{lemma}

\begin{proof} Let $P_G(x,y)$ be a path of $G$ connecting the vertices $x$ and $y$ outside the ball $B_k(s),$ where $k=d_G(s,x)-1.$
Let $P_G(x,s)$ and $P_G(y,s)$ be two shortest paths of $G$ connecting
the vertices $x,s$ and $y,s,$ respectively.
Let $c\in V(T)$ be the unique vertex in $T$ that is on the intersection
 $P_T(x,y) \cap P_T(x,s), \cap P_T(y,s).$
Since $c$ belongs to each of the paths $P_T(x,y),P_T(x,s),$ and $P_T(y,s),$   applying Lemma
\ref{lm:distance_to_path} three times, we infer that $d_T(c,P_G(x,y))\le \lambda/2,$ $d_T(c,P_G(x,s))\le \lambda/2,$ and $d_T(c,P_G(y,s))\le \lambda/2.$

Let $a$ be a closest  to $c$ vertex of $P_G(x,s)$ in the tree $T,$
i.e., $d_T(a,c)=d_T(c,P_G(x,s))\le \lambda/2.$  Let $z$ be a closest
to $a$ vertex of $P_G(x,y)$  in  $T.$ From condition (1) and
previous inequalities we conclude that $d_G(a,z)\le
d_T(a,z)=d_T(a,P_G(x,y))\le d_T(a,c)+d_T(c,P_G(x,y))\le \lambda.$
Since $z\in P_G(x,y)$ and $P_G(x,y)\cap B_k(s)=\emptyset,$
necessarily $d_G(s,z)\ge d_G(s,y)=d_G(s,a)+d_G(a,x),$ yielding
$d_G(a,x)\le d_G(a,z)\le \lambda.$  Analogously, if $b$ is a closest
to $c$ vertex of $P_G(y,s)$ in $T,$ then $d_G(b,y)\le \lambda$ and
$d_T(b,c)\le \lambda/2.$ By non-contractibility condition (1) and
triangle condition, $d_G(a,b)\le d_T(a,b)\le d_T(a,c)+d_T(b,c)\le
\lambda.$ Summarizing, we obtain the desired  inequality
$d_G(x,y)\le d_G(x,a)+d_G(a,b)+d_G(b,y)\le 3\lambda.$
\end{proof}

From Lemma \ref{lm:treeH} and Corollary \ref{col:treeH'} we immediately conclude

\begin{corollary} If a graph $G=(V,E)$ $\lambda$-embeds into a  tree,  then there exists an unweighted tree $H=(V,F)$ (without
Steiner points) and a $\{0,1\}$-weighted tree $H'=(V\cup
S',F')$ (with Steiner points), both constructible in linear
$O(|V|+|E|)$ time, such that $$d_H(x,y)-2 \leq d_G(x,y)\leq
d_H(x,y)+3 \lambda$$
and
$$d_{H'}(x,y)\leq d_{G}(x,y)\leq
d_{H'}(x,y)+3 \lambda$$ for any vertices $x,y\in V$.
\end{corollary}

This corollary shows that, for any unweighted graph $G$, it is
possible to turn its non-contractive multiplicative low-distortion
embedding into a weighted tree to a non-expanding additive
low-distortion embedding into a $\{0,1\}$-weighted tree. This seems
to be an interesting result on its own (note that the additive distortion of embedding
general finite metrics into trees can be approximated within a factor of 3
\cite{AgBaFaNaPa,ChFi}).

Since the largest diameter  $D$  of a cluster in ${\mathcal LP}$
can be computed in at most $O(|V| |E|)$ time, from Corollaries
\ref{col:treeHw} and \ref{cor:H'w}, we obtain:

\begin{corollary} \label{col:main-} If a graph $G=(V,E)$ $\lambda$-embeds into a  tree, then there exists
a uniformly weighted tree $H_w=(V,F,w)$
(without Steiner points) and a uniformly weighted tree
$H'_w=(V\cup S',F',w)$ (with Steiner points), both constructible in
$O(|V| |E|)$ time, such that $$d_G(u,v)\leq d_{H_w}(u,v)\leq
(3\lambda+1)(d_G(u,v)+2)$$
and
$$d_{G}(u,v)\leq
d_{H'_w}(u,v)\leq  (3\lambda+1) (d_{G}(u,v)+1)$$ for any vertices
$u,v$ of $G$.
\end{corollary}

Note that, although the topologies $H$ and $H'$ of trees $H_w$ and
$H'_w$ can be constructed in linear $O(|V|+|E|)$ time, we need to
compute the weights $w=D +1$ and $w=(D +1)/2$ assigned to
each edge of $H$ and $H',$ and this requires $O(|V| |E|)$ time.

Corollary \ref{col:main-} implies already that there exists a factor
12 approximation algorithm (resp., factor 8 approximation if Steiner points are used)
for the problem of non-contractive embedding an unweighted graph
into a tree with minimum multiplicative distortion. Below we show
that, by strengthening the result of Lemma \ref{lm:diam},
one can improve the approximation ratio from 12 to 9 and from 8 to
6.

\begin{lemma} Assume that $G=(V,E)$ $\lambda$-embed into a tree $T$, let $C=L^i_j\in {\mathcal LP}$
be a cluster of a layering partition of
$G$ and $v$ be an arbitrary vertex of $C$. Then, for any neighbor
$v'\in L^{i-1}$ of $v$   and any vertex $u\in C$, we have $d_G(v',u)\leq
\max\{3\lambda-1,2\lambda+1\}$.
\end{lemma}

\begin{proof}
Let $c \in V(T)$ be the nearest common ancestor in the tree $T$
(rooted at $s$) of all vertices of cluster $C=L^i_j$. Let $x$ and
$y$ be two vertices of $C$ separated by $c$. Let $P_G(x,y)$ be a
path of $G$ connecting vertices $x$ and $y$ outside the ball
$B_{i-1}(s).$  Then, as in the proof of Lemma \ref{lm:diam}, we have
$d_T(c,P_G(x,y))\le\lambda/2.$

Pick an arbitrary vertex $v\in C$ and a shortest path $P_G(v,s)$
connecting $v$ with $s$ in $G$. Since $c$ separates $v$ from $s$ in
$T$, by Lemma \ref{lm:distance_to_path}, $d_T(c,P_G(v,s))\leq \lambda/2$
holds.  Let $a_v$ be a closest to $c$ vertex of $P_G(v,s)$  in the tree $T$.
Then, $d_T(a_v,P_G(x,y))\leq d_T(a_v,c)+d_T(c,P_G(x,y))\leq
\lambda$. The choice of the path $P_G(x,y)$ and inequality (1) imply that
$d_G(a_v,v)\leq d_G(a_v,P_G(x,y))\leq d_T(a_v,P_G(x,y))\leq \lambda$.

Consider now an arbitrary vertex $u\in C$, $u\neq v$. By the
triangle inequality and (1), we have $d_G(a_v,a_u)\leq
d_T(a_v,a_u)\leq d_T(a_v,c)+d_T(a_u,c)\leq \lambda$ and, therefore,
$d_G(a_v,u)\leq d_G(a_v,a_u)+d_G(a_u,u)\leq 2\lambda$.

Let $v'\in L^{i-1}$ be a neighbor of $v$ in $P_G(v,s)$. If $a_v=v$,
then $d_G(v,u)=d_G(a_v,u)\leq 2\lambda$, i.e., $d_G(v',u)\leq d_G(
v,u)+1\leq 2\lambda+1$. Otherwise, if $a_v\neq v$, then
$d_G(v',u)\leq d_G(v',a_v)+d_G(a_v,u)\leq
\lambda-1+2\lambda=3\lambda-1,$ establishing the required inequality
$d_G(v',u)\leq
\max\{3\lambda-1,2\lambda+1\}$.
\end{proof}

To make the embedding of $G$ into the tree $H$ non-contractive, it suffices
to assign the same length $\ell:=\max\{3\lambda-1, 2\lambda+1\}$ to each
edge of $H$ and get a uniformly weighted tree $H_{\ell}=(V,F,\ell)$.

\begin{corollary} \label{col:treeHell} For any vertices $u,v$ of a graph $G$ which $\lambda$-embeds
into a tree, we have
$d_G(u,v)\leq d_{H_{\ell}}(u,v)\leq  \max\{3\lambda-1, 2\lambda+1\}
(d_G(u,v)+2).$
\end{corollary}

The tree $H_{\ell}$ provides a 9-approximation  to the problem
of non-contractive embedding an unweighted graph into a tree with
minimum multiplicative distortion. Note that the tree $H_{\ell}$ does
not have Steiner points. If we allow Steiner points, a better
approximation can be achieved. For this, we simply assign the same length
$\ell:=\frac{3\lambda}{2}$ to each edge of $H'$ and get a uniformly
weighted tree $H'_{\ell}.$

\begin{corollary} \label{col:treeHell-2} For any
vertices $u,v$ of a graph $G$ which $\lambda$-embeds
into a tree, we have
$d_G(u,v)\leq d_{H'_{\ell}}(u,v)\leq  3\lambda (d_G(u,v)+1).$
\end{corollary}

For a given graph $G=(V,E),$ we do not know $\lambda$ in advance,
however we know from Lemma \ref{lm:diam} that $\lambda^*(G,{\mathcal
T})\ge D/3.$ Therefore, the length $\ell$, which needs to be
assigned to each edge of the tree $H$ (which is defined in a
canonical way, independently of the value of $\lambda$), can be
found as follows: $\ell=\max\{d_G(u,v): uv \mbox { is an edge of }
H\}$. The length $\ell$, which needs to be assigned to each edge of
tree $H'$, can be found as follows:
$\ell=\frac{1}{2}\max\{D,\max\{d_G(u,v): uv \mbox{ is an edge of }
H\}\}$. Hence, $\ell$ can be computed in $O(|V| |E|)$ time. Our main
result of this section is the following algorithm and theorem.

{\small \vspace*{-2mm} {\footnotesize 
\begin{center}
\framebox{
\parbox{12cm}{
\vspace{0.05cm}
\noindent{\bf Algorithm} {\sc Approximation by Tree Metric}\\
{\footnotesize
 {\bf Input:} A graph $G=(V,E)$, a root vertex $s$ and  the corresponding layering partition  ${\mathcal LP}=\{ L_1^i,\ldots,L^i_{p_i}: i=0,1,\ldots,r\}$ of $G$\\
 {\bf Output:} Trees  $H$,  $H'$, $H_{\ell},$ and $H'_{\ell}$ for $G$\\
 \begin{tabular}[t!]{l@{ }p{14cm}}
 1. & Set initially  $H:=H':=(V,\emptyset)$. \\
 2. & {\bf For} $i=r$ {\bf downto}  1 {\bf do}\\
 3. & \begin{tabular}[t]{@{}p{2mm}@{}p{9cm}}& {\bf For} each cluster $C$ from $\{L^i_1,\ldots,L^i_{p_i}\}$ {\bf do}\end{tabular}\\
 4. & \begin{tabular}[t]{@{}p{5mm}@{}p{9cm}}& Pick a vertex $x_C$ in $L^{i-1}$ which has a neighbor in $C$. \end{tabular}\\
 5. & \begin{tabular}[t]{@{}p{5mm}@{}p{9cm}} & Add to $H$ edges $\{vx_C: v\in C\}$. \end{tabular}\\
 6. & \begin{tabular}[t]{@{}p{5mm}@{}p{10cm}} & Add to $H'$ a Steiner point $p_C$ and edges $\{vp_C: v\in C\}$ and $p_Cx_C$. \end{tabular}\\
 7. & Set $\ell:=\max\{d_G(u,v): uv \mbox { is an edge of } H\}$. \\
 8. & Set $H_{\ell}:= H$ and assign length $\ell$ uniformly to all edges of $H_{\ell}$. \\
 9. & Set $\ell:=\frac{1}{2}\max\{D,\ell\}$, where $D$ is the largest diameter of a cluster from ${\mathcal LP}$.\\
 10. &   Set $H'_{\ell}:= H'$ and assign length $\ell$ uniformly to all edges of $H'_{\ell}$. \\
 11. & Return trees $H$,  $H'$, $H_{\ell}$ and $H'_{\ell}$. \\
 \end{tabular}}
\vspace{-0.3cm}\\ }}\hspace{0.2cm}
\end{center}
}}

\begin{theorem}
There exists a factor 6 approximation algorithm with running time $O(|V| |E|)$ for
the optimal multiplicative distortion
$\lambda^*(G,{\mathcal T})$ of non-contractive embedding an
unweighted graph $G$ into a tree.
\end{theorem}

Our 6-approximation algorithm improves previously known
100-approximation \cite{BaInSi} and 27-approximation
\cite{BaDeHaSiZa} algorithms. In fact, the approximation ratio 6
holds only for adjacent vertices of $G$. It decreases when distances
in $G$ increase. For vertices at distance $\geq 2$, the ratio is $\leq 4.5$.
For vertices at distance $\geq 3$, the ratio is $\leq 4$.
Our tree $H_{\ell}$ does not have any Steiner points and the edges
of both trees $H_{\ell}$ and $H'_{\ell}$ are uniformly weighted. The
tree $H'_{\ell}$, with Steiner points, is better than the tree
$H_{\ell}$ only for small graph distances. So, the Steiner points do
not really help, confirming A. Gupta's claim \cite{GuptaSt}.

Our technique works also in more general cases. In particular, if an

unweighted graph $G=(V,E)$ admits a weighted tree $T=(V\cup S,U)$
with
$$d_G(x,y)\leq \alpha \cdot d_T(x,y)+\beta  ~~~~\mbox{ and }~~~~ d_T(x,y)\leq \lambda\cdot d_G(x,y)+\delta$$
for all $x,y\in V$, then each cluster of a layering partition of $G$
has diameter at most $3(\alpha(\lambda+\delta)+\beta)$. Moreover,
$H=(V,F)$, weighted appropriately, will give a good approximation of
$T$.

\section{Minors, relaxed minors, and metric minors}

In this section, we introduce the notions of relaxed minors and $\alpha$-metric relaxed minors, which,
together with layering partitions, are used in the algorithm for
approximating the optimal distortion of embedding unweighted graphs into
outerplanar graphs (i.e., $K_{2,3}$-minor free graphs approximation). These concepts and results
may be helpful for designing approximation algorithms for embedding graphs into other classes
of minor closed graphs.

\subsection{Minors and relaxed minors}

Recall that a graph $H$ is a {\it minor} of a graph  $G$ if a graph isomorphic to $H$ can be obtained from $G$ by
contracting some edges, deleting some edges, and deleting some isolated vertices \cite{Di}. Notice that the property of being minor is transitive, i.e.,
if $G'$ is a minor of $G$ and $H$ is a minor of $G',$ then $H$ is a minor of $G.$ To adapt the concept of  minor to our
embedding purposes, note that $H=(V',E')$ is a  minor of $G=(V,E)$ provided there exists a map $\mu: V'\cup E'\mapsto 2^V,$ such that

\begin{itemize}
\item[(i)] for any vertex $v$ of $H,$ $G(\mu(v))$ is connected;
\item[(ii)]for any different vertices $v,v'$ of $H,$ $G(\mu(v))\cap G(\mu(v'))=\emptyset;$
\item[(iii)] for any edge $e=uv$ of $H,$ $G(\mu(e))$ is a path $P_e$ of $G$ with one end in $G(\mu(u))$ and  another end in $G(\mu(v));$
\item[(iv)] for any vertex $v$ and any edge $e$ of $H$ with $v\notin e,$ $P_e\cap G(\mu(v))=\emptyset;$
\item[(v$'$)] for any two edges $e=(x,y),e'=(u,v)$ of $H$, the paths $P_e$ and
$P_{e'}$ intersect if and only if $\{x,y\} \cap \{u,v\} \neq
\emptyset$ and if say, $e=(x,y), e'=(x,w)$ then $P_e$ and $P_{e'}$
intersect only in $\mu(x)$.
\end{itemize}

\noindent Indeed, if such a map $\mu$ exists, then contracting each
connected subgraph $\mu(v), v\in V',$ to a single vertex $v$ and
each path $P_e$ to an edge $e$, the conditions (ii),(iii), and
(v$'$) ensure that the resulting graph  will be isomorphic to $H,$
i.e., $H$ is indeed a minor of $G.$ Note that if in (v$'$) two paths
$P_e$ and $P_{e'}$ intersect, then they intersect in the subgraph
$G(\mu(u)),$ where $u$ is the common end of $e$ and $e'.$ In
particular, if the edges $e,e'$ are non-incident, then the paths
$P_e$ and $P_{e'}$ are disjoint.

For our metric related theory we will need a weaker notion of minor by allowing intersecting paths
to intersect anywhere.  A graph $H=(V',E')$ is called a {\it relaxed minor} of a graph $G=(V,E)$
if there exists a map $\mu: V'\cup E'\mapsto 2^V$ satisfying the conditions (i)-(iv) and
the following relaxation of (v$'$):

\begin{itemize}
\item[(v)] for any two non-incident edges $e,e'$ of $H$, the paths $P_e$ and $P_{e'}$ are disjoint.
\end{itemize}

\noindent
The concept of relaxed minor is weaker than that of minor. For example,  the triangle $C_3$ (3-cycle) is
not a minor of any tree, but it is a relaxed minor of the star $K_{1,3}:$ $\mu$ maps the three
vertices of $C_3$ to the three leaves of $K_{1,3}$ and maps each edge
$uv$ of $C_3$ to the unique path of $K_{1,3}$ between the leaves
$\mu(u)$ and $\mu(v).$ The map $\mu$ satisfies the conditions (i)-(v)
but does not satisfy the condition (v$'$).

Relaxed  and $\alpha$-metric relaxed minors (see
Subsection \ref{sec:metric-relaxed})
are crucial because their existence corresponds to a witness that  $G$  {\em cannot} be embedded into
$H$-relaxed-minor-free graphs with small distortion (see Proposition
\ref{metric-relaxed-minor}). Thus it seems important to relate this notion to
standard minors. We conjecture that {\it if the graph $H$ is triangle-free, then the notion of relaxed
minor is not weaker than that of minor.} Here  we
prove a weaker statement. We note that while this leaves this graph
theoretic point not settled, it has no bearing regarding the metric
consequences (up to a factor of $2$ in the distortion lower bound,
of Proposition \ref{metric-relaxed-minor}). We
established a weaker statement which is enough to deal with $H$ of
special form: $H$ will be bipartite $H=(V,F;E)$ with
every vertex $f \in F$ of degree two. Such {\em subdivided} graphs $H$ can be seen as a
subdivision of an arbitrary graph $H'= (V,E')$ where $(u,v) \in H'$ iff there is a member $f \in
F$ such that $(u,f), (v,f) \in E$.

The notion of relaxed minors and in particular, its metric
strengthening to $\alpha$-metric relaxed minor (see
Subsection \ref{sec:metric-relaxed})
is crucial in the
discussion above. Its existence corresponds to a witness that the
corresponding $G$-metric {\em cannot} be embedded into
$H$-relaxed-minor-free graphs with small distortion (see Proposition
\ref{metric-relaxed-minor}). Thus it seems important to relate this notion to
standard minors. We conjecture that if the graph $H$ is triangle-free, then the notion of relaxed
minor is not weaker than that of minor. Here  we
prove a weaker statement. We note that while this leaves this graph
theoretic point not settled, it has no bearing regarding the metric
consequences (up to a factor of $2$ in the distortion lower bound,
of Proposition \ref{metric-relaxed-minor}).



\begin{proposition}\label{relaxed-minor} If a graph $G=(V,E)$ has a subdivided graph $H=(V',E')$ as a relaxed minor, then $G$ has $H$ as a minor.
\end{proposition}

\begin{proof} We proceed by induction on the total number of vertices  and edges of the graph $G.$
The base case for which $H=G$ is trivial. Let $H$ be a relaxed minor of $G$ and let
$\mu:V'\cup E'\mapsto 2^V$ be the map  satisfying the conditions
(i)-(iv) and (v). Suppose, by way of contradiction, that $H$ is not
a minor of $G,$ in particular, that $\mu$ does not satisfy the
condition (v$'$). For each edge $uv$ of $H,$ we will denote by
$u^*\in \mu(u)$ and $v^*\in \mu(v)$ the end vertices of the path
$P_e.$ Note that we can suppose that each such  path $P_e, e\in E',$
intersects the connected subgraphs $G(\mu(u))$ and $G(\mu(v))$ in a
single vertex. Indeed, if say $P_e$ intersects  $G(\mu(u))$ and/or
$G(\mu(v))$ in several vertices, then let $x$ be the last vertex of
$P_e\cap \mu(u)$ while moving along $P_e$ from $u^*$ to $v^*$ and
let $y$ be the first vertex of $P_e\cap \mu(v)$ while moving from
$x$ to $v^*.$  Replacing in the definition of $\mu$ the path $P_e$
by its subpath $P'_e$ between $x$ and $y,$ we will obtain a map
$\mu$ which still satisfies the definition of relaxed minors and
such that $|P'_e\cap \mu(u)|=1$ and $|P'_e\cap \mu(v)|=1.$ So, we
will further assume that $\mu$ obeys this additional condition,
i.e., $P_e\cap \mu(u)=\{ u^*\}$ and $P_e\cap \mu(v)=\{v^*\}$ for any
edge $e=uv$ of $H$.

We assert now that for each vertex $v$ of the graph $H,$ the
subgraph of $G$ induced by  $\mu(v)$ consists of a single vertex. If
this is not the case, then let $G'$ be the graph obtained from $G$
by contracting the connected subgraph $G(\mu(v))$ to a single vertex
$v'.$ Then, $G'$ is a minor of $G.$ Denote by $\psi$ the map from
$G$ to $G'$ defining this contraction, i.e., $\psi(u)=u$ if $u\notin
\mu(v)$ and $\psi(u)=v',$ otherwise. Then, the composition of $\mu$
with $\psi$ is a map from $H$ to $G'$ satisfying the conditions
(i)-(v), i.e., $H$ is a relaxed minor of $G'.$ By induction
assumption, $H$ is a minor of $G'$ and therefore must be a minor of
$G$ as well, contrary to our assumption.  Therefore, for each vertex
$v$ of $H,$ the set $\mu(v)$ consists of a single vertex of $G,$
which we will further denote by $v^*$.  We can also suppose that the
paths $P_e, e\in E',$ are induced paths of $G,$ otherwise we can
shortcut them without violating the conditions (i)-(v). Similarly, it
is easy to observe that no edge of $G$ is used by more than one path
of the form $P_e$, as otherwise, such an edge can be contracted
leaving $H$ as a relaxed minor in the contracted graph.

Since $H$ is a subdivided graph, $H$ is of the form $(V,F;E),$ where
each vertex  $f \in F$ has degree 2. Since $\mu$ does not satisfy
the condition (v$'$) for minors, there exist two incident edges
$e=uv$ and $e'=uw$ of $H$ such that the paths $P_e$ and $P_{e'}$ of
$G$ intersect in other vertices  except the vertex $u^*=\mu(u).$
Suppose first that this happens for some $u=f \in F$. Namely, $P_e$
and $P_{e'}$ intersect in $u^*$ and in addition in some $x \in
V(G)$, where $x$ is the closest along $P_e$ to $u^*$. It is easy to
see that one can change $\mu$ so that to map $\mu(u)$ to $x$, to map
$e$ to the suffix of $P_e$ from $x$ to $v^*,$ and $P_{e'}$ to the
suffix from $x$ to $w^*$, while still having $H$ as a relaxed minor
of $G$. However, now the first edge $(u^*,u')$ in the former path
$P_e$ is not used anymore and can be deleted, thus induction ends
the proof. We conclude that for every $f \in F$ the two paths that
corresponds to the edges adjacent to $f$, intersect only in their
ends points $\mu(f)$.

Assume then, that there is vertex in $H$, $u \in V$ and two edges $e_1=(u,f_1),
e_2=(u,f_2)$ such that the paths $\mu(e_1) = P_1$ and
$\mu(e_2)=P_2$ intersect also at $x \neq u^*$. Let $u^*u_1,
u^*u_2$ be the first edges in the paths $P_1, P_2$ respectively,
and note that $u_1, u_2  \neq x$ otherwise, if say $u_1 =x$, then
the edge $u^*x$ could be contracted (and the map $\mu$ changed
accordingly), preserving $H$ as a relaxed minor of the smaller
resulting graph, implying the result.

Now, the edge $(u^*,u_1)$ cannot be contracted only if $u_1$ is used by
another path $\mu(e_3) = P_3$ for some $e_3 \in H$ that is not
adjacent to $u$ (this is allowed as $H$ is a relaxed minor). However,
as $P_3$ intersects $P_1$, it must be the case that $e_3$ is adjacent
to $e_1$ at $f_1$. This, however, contradicts our conclusion before
that the two paths adjacent to any $\mu(f), ~f \in F$ intersect only at
their end points. This contradiction completes the proof of the Proposition.
\end{proof}

\subsection{$\alpha$-Metric relaxed minors}\label{sec:metric-relaxed} We say that two sets $A,B$ of a graph $G$ are $\alpha$-{\it far} if $\min\{ d_G(a,b): a\in A, b\in B\}> \alpha.$ For $\alpha\ge 1,$ we call a graph $H=(V',E')$ an $\alpha$-{\it metric relaxed minor} of a graph $G=(V,E)$ if there exists a map
$\mu: V'\cup E'\mapsto 2^V$ satisfying the conditions (i)-(v) (i.e., $H$ is a relaxed minor of $G$) and the following stronger version of condition (v):

\begin{itemize}
\item[(v$^+$)] for any two non-incident edges $e=uv$ and $e'=u'v'$ of $H$, the sets $\mu(u)\cup P_e\cup \mu(v)$ and $\mu(u')\cup P_{e'}\cup \mu(v')$ are $\alpha$-far in $G$.
\end{itemize}
\noindent
To motivate the concept of  $\alpha$-metric relaxed minor, we  establish first the following basic property of  embeddings with (multiplicative) distortion $\le\alpha$ of unweighted graphs $G$ into (possibly weighted) graphs $G'$.

Let $\varphi$ be an embedding of a graph $G=(V,E)$ into a graph $G'=(V',E')$ having distortion at most $\alpha.$ For a set $S\subseteq V$ inducing a connected subgraph $G(S)$ of $G,$ we denote by $[\varphi(S)]$ a union of shortest paths of $G'$ running between each pair of vertices of $\varphi(S)$ which are images of adjacent vertices of $G(S),$ one shortest path per pair.

\begin{lemma}\label{basic-fact-distortion} If a graph $G$ $\alpha$-embeds into a graph $G'$ and two edges $e_1=a_1a_2$ and $e_2=b_1b_2$ are $\alpha$-far in $G,$ then $[\varphi(e_1)]\cap [\varphi(e_2)]=\emptyset.$ More generally, if two sets of vertices $A,B$ induce connected subgraphs of $G$ and are $\alpha$-far, then $[\varphi(A)]\cap [\varphi(B)]=\emptyset.$
\end{lemma}

\begin{proof} For a vertex $v$ of $G,$ let $v^*=\varphi(v).$ Suppose, by way of contradiction, that the shortest paths $P_{e_1}=[\varphi(e_1)]$ between $a^*_1,a^*_2$ and
$P_{e_2}=[\varphi(e_2)]$ between $b^*_1,b^*_2$ intersect in a vertex $x.$ Since
$$1=d_G(a_1,a_2)\le d_{G'}(a^*_1,a^*_2)\le \alpha\cdot d_G(a_1,a_2)=\alpha, ~~~~d_{G'}(a^*_1,a^*_2)=d_{G'}(a^*_1,x)+d_{G'}(x,a^*_2),$$
$$1=d_G(b_1,b_2)\le d_{G'}(b^*_1,b^*_2)\le \alpha\cdot d_G(b_1,b_2)=\alpha,  ~~~~d_{G'}(b^*_1,b^*_2)=d_{G'}(b^*_1,x)+d_{G'}(x,b^*_2),$$
we conclude that
$$\min\{ d_{G'}(a^*_1,x),d_{G'}(x,a^*_2)\}\le \alpha/2 \mbox{ and } \min \{ d_{G'}(b^*_1,x),d_{G'}(x,b^*_2)\}\le \alpha/2.$$
Suppose, without loss of generality, that $d_{G'}(a^*_1,x)\le
\alpha/2$ and $d_{G'}(b^*_1,x)\le \alpha/2.$ Since $d_G(a_1,b_1)\le
d_{G'}(a^*_1,b^*_1)\le d_{G'}(a^*_1,x)+d_{G'}(b^*_1,x)\le
\alpha/2+\alpha/2\le \alpha,$ we obtain a contradiction with the
assumption that the edges $e_1=a_1a_2$ and $e_2=b_1b_2$ are
$\alpha$-far in $G.$

To establish the second assertion, suppose, by way of contradiction,
that $[\varphi(A)]\cap [\varphi(B)]\ne\emptyset.$ From the
definition of the sets $[\varphi(A)]$ and $[\varphi(B)]$ we conclude
that $[\varphi(e_1)]\cap [\varphi(e_2)]\ne\emptyset$ for an edge
$e_1$ of $G(A)$ and an edge $e_2$ of $G(B).$ From the first part of
the proof, we know that the edges $e_1$ and $e_2$ cannot be
$\alpha$-far, thus the sets $A$ and $B$ cannot be $\alpha$-far
either.
\end{proof}

We will show now  that under some general conditions on $H,$ the presence in a graph $G$  of an $\alpha$-metric relaxed minor isomorphic to $H$
is an obstacle for embedding $G$ into a $H$-minor free graph with distortion  at most $\alpha.$

\begin{proposition}\label{metric-relaxed-minor} If a subdivided 2-connected graph $H=(V',E')$ is an $\alpha$-metric relaxed minor of a
graph $G=(V,E),$ then any embedding of $G$ into an $H$-minor free graph requires distortion $>\alpha.$
\end{proposition}

\begin{proof} Suppose, by way of contradiction, that $G$ has an embedding $\varphi$ with distortion $\le \alpha$ into an $H$-minor free graph $G'.$
Let $\mu: V'\cup E'\mapsto 2^V$ be a map showing that $H$ is an
$\alpha$-metric relaxed minor of $G.$ Before deriving a
contradiction with this assumption, we consider some properties of
maps $\varphi$ and $\mu$. First note that we can extend $\varphi$
from the vertex-set $V$ of $G$ to the edge-set $E$ by associating
with each  edge $e$ of $G$ the  shortest path $P_e:=[\varphi(e)]$ of
$G'.$ Pick any vertex $v$ of $H.$ Then, $\varphi(\mu(v))$ is a
connected subgraph of $G'$ because each of the maps $\mu$ and
$\varphi$ maps connected subgraphs to connected subgraphs. Moreover,
from Lemma \ref{basic-fact-distortion} we know that $\varphi$ maps
two $\alpha$-far connected subgraphs of $G$ to two disjoint
subgraphs of $G'.$ As to the map $\mu,$  we assert that it satisfies
the following two conditions:
\begin{itemize}
\item[(ii$^+$)] for any  two different vertices $v,v'$ of $H,$ the sets $\mu(v)$ and $\mu(v')$ are $\alpha$-far;
\item[(iv$^+$)] for any vertex $v$ and any edge $e$ of $H$ with $v\notin e,$ the sets $\mu(v)$ and $\mu(e)=P_e$ are $\alpha$-far.
\end{itemize}
Since $H$ is 2-connected, any two distinct vertices $v,v'$ belong to
a common cycle of $H.$  Since $H$ is triangle-free, $v$ and $v'$
belong to two non-incident edges $e,e'$ of this cycle.  Applying
property (v$^+$) to $e$ and $e',$ we conclude that $\mu(v)$ and
$\mu(v')$ are $\alpha$-far, establishing (ii$^+$). Analogously for
(iv$^+$), if $v\notin e$ then, by 2-connectivity of $H$, we can find
a cycle passing via $v$ and $e.$ Since $G$ is triangle-free, one of
two edges of this cycle containing $v$, say $e'$,  is not incident
to $e.$ Again, applying the condition (v$^+$) to the edges $e$ and
$e',$ we conclude that the sets $\mu(v)$ and $P_e$ are $\alpha$-far,
establishing (iv$^+$).

 Now, we define the following map $\nu:  V'\cup E'\mapsto 2^{V(G')}$ from $H$ to $G'.$ For each vertex $v\in V',$ we set $\nu(v)=\varphi(\mu(v)).$
 For each edge $e=uv$ of $H,$ $\mu(e)=P_e$ is a path of the graph $G$ with end-vertices $u^*\in \mu(u)$ and $v^*\in \mu(v).$ Each edge $f$ of $P_e$ is
 mapped by $\varphi$ to a path $\varphi(f)$ of $G'.$ Define $\nu(e)$ to be any path of $G'$ between the vertices $u'=\varphi(u^*)$ and $v'=\varphi(v^*)$ contained in the set $\bigcup\{ \varphi(f): f \mbox{ is an edge of } P_e\}$.  From definition of $\nu$ and properties of $\mu$ and $\varphi$ it immediately follows that  the map $\nu$ satisfies the conditions (i) and (iii). We will show now that $\nu$ satisfies the conditions (ii), (iv), and (v) as well. To verify (ii), pick two distinct vertices $u,v$ of $H.$

 By condition (ii$^+$), the sets $\mu(u)$ and $\mu(v)$ are $\alpha$-far, thus the second assertion of  Lemma \ref{basic-fact-distortion} implies that the sets $\nu(u)=\varphi(\mu(u))$ and $\nu(v)=\varphi(\mu(v))$ are disjoint, thus showing (ii). Analogously,
 if  $v$ is a vertex and $e$ is an edge of $H$ with $v\notin e,$ then, by (iv$^+$), the sets $\mu(v)$ and $P_e=\mu(e)$ are $\alpha$-far, thus,
 by Lemma \ref{basic-fact-distortion}, the sets $\nu(v)=\varphi(\mu(v))$ and $\varphi(P_e)$ are disjoint. Since $\nu(e)\subseteq \varphi(P_e),$ the sets $\nu(v)$ and $\nu(e)$ are disjoint as well, establishing (iv).
 The last condition (v) can be derived in a similar way by using (v$^+$) and Lemma \ref{basic-fact-distortion}. Hence, the map $\nu$ satisfies the conditions (i)-(v), thus $H$ is a relaxed minor of $G'.$
 Since $H$ is triangle-free, by Proposition \ref{relaxed-minor},  $H$ is a minor of $G',$ contrary to the assumption that the graph $G'$ is $H$-minor free. This concludes the proof
 of the proposition.
\end{proof}

\subsection{Lower bounds for $\alpha$-embeddings into $K_{2,r}$-minor free graphs} We will use the results of previous section to give lower bounds for the multiplicative
distortion of embedding an unweighted graph $G=(V,E)$ into $K_{2,r}$-minor free (possibly weighted) graphs.

\begin{proposition} \label{K_2,r} If for $\alpha>1$ a cluster $C$ of a layering partition ${\mathcal LP}$ of a graph $G$ contains $r\ge 3$
vertices $v^*_1,\ldots,v^*_r$ that are pairwise $(4\alpha+2)$-far, then any embedding $\varphi$ of $G$ into a $K_{2,r}$-minor free graph
has distortion $>\alpha.$
\end{proposition}

\begin{proof} Suppose that the layering partition ${\mathcal LP}$ of $G$ was defined with respect to the
vertex $s$ and let $T$ be a BFS tree rooted at $s.$ Let $k$ denote
the distance from $s$ to any vertex of  the cluster $C.$ Since $C$
contains  $(4\alpha+2)$-far vertices $v^*_1,\ldots,v^*_r,$ we
conclude that $k\ge 2\alpha+2.$ We will define now a mapping $\mu$
from $K_{2,r}$ to $G$ which will allow us to conclude that $K_{2,r}$
is an $\alpha$-metric relaxed minor of $G.$ Then, since $K_{2,r}$ is
2-connected and triangle-free, Proposition
\ref{metric-relaxed-minor} will show that any embedding of $G$ into
a $K_{2,r}$-minor free graph has distortion $>\alpha.$ Denote by
$u_1,\ldots,u_r,v,w$ the vertices of $K_{2,r},$ where $v$ and $w$
are the two vertices of degree $r.$ Finally, denote by $e_i$ the
edge $vu_i$ and by $f_i$ the edge $wu_i,$ $i=1,\ldots,r.$

Let $P_1,\ldots, P_r$ be the paths of the tree $T$ of length
$\alpha+1$ from the vertices $v^*_1,\ldots,v^*_r,$ respectively,
towards the root $s.$ Denote by $u^*_1,\ldots,u^*_r$ the other end
vertices of the paths $P_1,\ldots,P_r.$ Let $R_1,\ldots,R_r$ be the
paths of $T$ of length $\alpha+1$ from the vertices
$u^*_1,\ldots,u^*_r$,  respectively, towards $s.$ Denote by
$w^*_1,\ldots,w^*_r$ the other end vertices of the paths
$R_1,\ldots,R_r.$ Set $\mu(u_i):=u^*_i,$  $\mu(e_i):=P_i$ and
$\mu(f_i):=R_i$ for $i=1,\ldots,r.$ Let $\mu(v)$ be the connected
subgraph of $G$ induced by all (or some) paths connecting the
vertices $v^*_1,\dots,v^*_r$ outside the ball $B_{k-1}(s).$ Finally,
let $\mu(w):=B_{k-2\alpha-2}(s)$ (clearly, $w^*_1,\ldots,w^*_r$
belong to $\mu(w)$); for an illustration, see Fig. \ref{fig_K_2,r}. From the definition of the map $\mu$ and of the
layering partition ${\mathcal LP}$, we immediately conclude that
$\mu$ satisfies the conditions (i) and (iii). We will show now that
$\mu$ also satisfies the conditions (ii$^+$),(iv$^+$), and (v$^+$).
Since $\mu(v)\subseteq \cup_{j\ge k} L^j,
\mu(w)=B_{k-2\alpha-2}(s),$ and the vertices
$u^*_1=\mu(u_1),\ldots,u^*_r=\mu(u_r)$ all belong to the sphere
$L^{k-\alpha-1},$ we conclude that the $\mu$-images  of the vertices
of $K_{2,r}$ are pairwise $\alpha$-far in $G,$ whence $\mu$
satisfies the condition (ii$^+$). Analogously, from the definition
of the layering of $G$ we conclude that any vertex of $\mu(v)$ is at
distance $>\alpha$ from any path $R_i=\mu(f_i)$ and any vertex of
$\mu(w)$ is at distance $>\alpha$ from any path $P_i=\mu(e_i).$ If a
vertex $u^*_i$ is at distance $\le \alpha$ from a vertex $x$ of
$P_j\cup R_j$ for $j\ne i,$ then, by triangle inequality, we obtain
$d_G(v^*_i,v^*_j)\le d_G(v^*_i,u^*_i)+d_G(u^*_i,x)+d_G(x,v^*_j)\le
\alpha+1+\alpha+d_G(v^*_j,x).$ Since $x\ne w^*_j,$ $d_G(v^*_j,x)\le
2\alpha+1,$ yielding  $d_G(v^*_i,v^*_j)\le
\alpha+1+\alpha+2\alpha+1=4\alpha+2,$ contrary to the assumption
that $v^*_i$ and $v^*_j$ are $(4\alpha+2)$-far. This contradiction
shows that $\mu$ satisfies the condition (iv$^+$). It remains to
show that $\mu$ also satisfies the condition (v$^+$), namely that
any two paths $P_i$ and $R_j$ with $i\ne j$ are $\alpha$-far. If
$d_G(x,y)\le \alpha$ for $x\in P_i\setminus \{ v^*_i,u^*_i\}$ and
$y\in R_j\setminus \{ u^*_j,w^*_j\},$ then $d_G(v^*_i,v^*_j)\le
d_G(v^*_i,x)+d_G(x,y)+d_G(y,v^*_j)\le \alpha+\alpha+2\alpha+1\le
4\alpha+1,$ contrary to the assumption that $v^*_i$ and $v^*_j$ are
$\alpha$-far. This contradiction shows that $\mu$ satisfies (v$^+$),
i.e., indeed $K_{2,r}$ is an $\alpha$-metric relaxed minor of $G.$
\end{proof}

\begin{figure}[t]
\begin{center}
\scalebox{0.4}
{\input{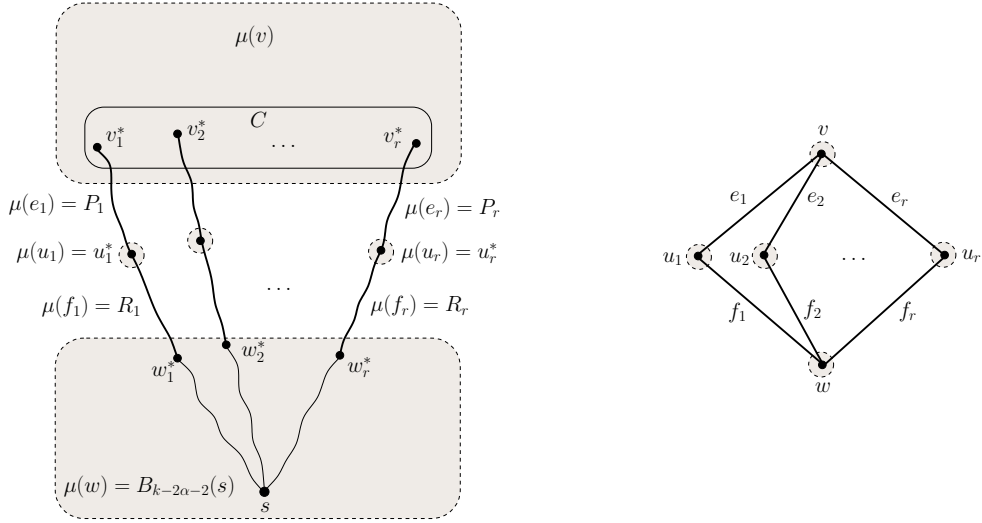}}
\end{center}
\caption{To the proof of Proposition \ref{K_2,r}}
\label{fig_K_2,r}
\end{figure}

Notice that outerplanar graphs are exactly the graphs which do not contain $K_{2,3}$ and $K_4$ minors.
From Proposition \ref{K_2,r} we immediately obtain the following corollary:

\begin{corollary} \label{K_2,3} If for $\alpha>1$ a cluster $C$ of a layering partition
 of a graph $G$ contains three vertices that are pairwise $(4\alpha+2)$-far,
then any embedding $\varphi$ of $G$ into an outerplanar graph  has distortion $>\alpha.$
\end{corollary}

\section{Approximation algorithm for embedding graph metrics  into  outerplanar graphs}
We present now the algorithm for constant-factor approximation of the distortion of
the best embedding of an unweighted graph into outerplanar metrics. Let $\lambda$ be the
best such multiplicative distortion for an input graph $G$. We first
study the structure of a layered partition of $G$.

\subsection{Small, medium, and big clusters}
Let $G=(V,E)$ be the input graph and consider a layering partition
${\mathcal LP}$ of $G$ into clusters. We assume that  $\lambda\ge 1$
is so that each cluster $C$ of ${\mathcal LP}$  contains at most two
vertices which are $(4\lambda+2)$-far (otherwise, by Corollary
\ref{K_2,3}, the optimal distortion  of embedding $G$ into an
outerplanar graph is larger than $\lambda$). Set
$\Lambda:=4\lambda+2.$ We call a cluster $C$ {\it bifocal} if it has
exactly two $\Lambda$-far vertices $c_1$ and $c_2.$  In addition, for
such cluster let $C_1=\{
x\in C: d_G(x,c_1)\le d_G(x,c_2)\}$ and $C_2=\{ x\in C:
d_G(x,c_2)\le d_G(x,c_1)\},$ and call $C_1$ and $C_2$ the {\it
cells} of $C$ centered at $c_1$ and $c_2,$ respectively (in what
follows,  we will suppose that $c_1$ and $c_2$ form a diametral pair
of $C$, i.e., $d_G(c_1,c_2)={\rm diam}(C)=\max\{d_G(u,v): u,v\in
C\}$). If ${\rm diam}(C)\le \Lambda$ (i.e., $C$ is not bifocal),
then the cluster $C$ is called {\it small}. Then $C$ has a unique
cell centered at an arbitrary vertex of $C.$  A bifocal cluster $C$
is called {\it big} if ${\rm diam}(C)>16\lambda+12,$ otherwise, if
$\Lambda<{\rm diam}(C)\le 16\lambda+12,$ then $C$ is called a {\it
medium} cluster. An {\it almost big cluster} is a medium cluster $C$
such that ${\rm diam}(C)>16\lambda+10.$  We say that a cluster $C$
is $\Delta$-{\it separated} if $C$ is bifocal with cells $C_1$ and
$C_2$ and $d_G(u,v)>\Delta$ for any $u\in C_1$ and $v\in C_2.$
Further, we will set $\Delta:=8\lambda+6.$ We say that a bifocal
cluster $C'$ is {\it spread} if both cells $C_1,C_2$ of its father
$C$ are adjacent to $C'.$ Finally, we say that two disjoint sets $A$
and $B$ are {\it adjacent} in $G$ if there exists an edge of $G$
with one end in $A$ and another end in $B.$

\begin{lemma} \label{bifocal_cluster} If $C$ is a bifocal cluster of
a layering partition ${\mathcal LP}$ of $G,$ then the diameter of
each of its cells $C_1$ and $C_2$ is at most $2\Lambda.$
\end{lemma}

\begin{proof} Let $x,y\in C_1.$ Clearly, $d_G(x,c_1)\leq 4\lambda+2$ and $d_G(y,c_1)\le 4\lambda+2.$
Therefore, by triangle inequality, $d_G(x,y)\le
d_G(x,c_1)+d_G(c_1,y)\le 8\lambda+4=2\Lambda.$
\end{proof}

\begin{lemma} \label{prop_big_cluster} If $C$ is a bifocal cluster of
a layering partition ${\mathcal LP}$ of $G$ such that ${\rm diam}(C)=d_G(c_1,c_2)>12\lambda+6,$ then $C$ has the
following properties:

\begin{itemize}
\item[(i)] $C$ is $({\rm diam}(C)-2\Lambda-1)$-separated, in particular $C_1\cap C_2=\emptyset;$
\item[(ii)] ${\rm diam}(C_1)\le \Lambda$ and ${\rm diam}(C_2)\le \Lambda.$
\end{itemize}
If $C$ is a big cluster, then $C$ is $(8\lambda+8)$-separated, and
if $C$ is an almost big cluster, then $C$ is
$(8\lambda+6)$-separated. In particular, big and almost big clusters
are $\Delta$-separated. If $C$ is a big or an almost big cluster,
then ${\rm diam}(C_1)\le \Lambda$ and ${\rm diam}(C_2)\le \Lambda.$
\end{lemma}

\begin{proof} Since the cluster $C$ is bifocal, from the definition of
its cells we conclude that, for any two vertices $u\in C_1$ and
$v\in C_2$, $d_G(u,c_1)\le 4\lambda+2$ and $d_G(v,c_2)\le
4\lambda+2.$ Therefore, $12\lambda+6<{\rm diam}(C)=d_G(c_1,c_2)\le
d_G(c_1,u)+d_G(u,v)+d_G(v,c_2)\le d_G(u,v)+8\lambda+4,$ showing that
$d_G(u,v)>{\rm diam}(C)-2\Lambda-1$ and $d_G(u,v)>4\lambda+2,$
whence $C$ is $({\rm diam}(C)-2\Lambda-1)$-separated as well as
$(4\lambda+2)$-separated. Furthermore, from  $d_G(u,v)\geq {\rm
diam}(C)-8\lambda-4$ we obtain that any big cluster (i.e., a cluster
$C$ with ${\rm diam}(C)>16\lambda+12$) is $(8\lambda+8)$-separated
and any almost big cluster (i.e., a cluster $C$ with
$16\lambda+10<{\rm diam}(C)\leq 16\lambda+12$) is
$(8\lambda+6)$-separated. If $C_1$ contains two vertices $x,y$ with
$d_G(x,y)>4\lambda+2,$ then  the vertices $x,y,$ and $c_2$ are
pairwise $(4\lambda+2)$-far, contradicting the assumption that $C$
is bifocal.
\end{proof}

Given a cluster $C$ located at distance $k$ from the root $s$ and
its son $C'$ (in the tree $\Gamma$), we call the union of $C$ with
the connected component of $G(V\setminus B_k(s))$ containing $C'$
the $CC'$-{\it fiber} of $G$ and denote it by ${\mathcal F}(C,C').$
Note that the son-father relation between clusters that we use here
and in what follows is with respect to tree $\Gamma$.

\begin{lemma} \label{big_cluster1} If a cluster $C$ of a layering
partition ${\mathcal LP}$ of $G$ is big, then $C$ has a son $C'$
which is a bifocal spread cluster such that contracting the four cells of
$C$ and $C'$ (but preserving the edges between different cells), we
will obtain a $2K_2,$ an induced  matching with two edges.
\end{lemma}

\begin{proof}  Let $C=C_1\cup C_2$ be the partition of $C$ into cells.
Pick $x\in C_1$ and $y\in C_2$ and consider a $xy$ path $P$ in the
subgraph of $G$ induced by $V\setminus B_{k-1}(s),$ where $k$ is the
distance from the root $s$ to all vertices of $C.$ Since $C$ is big,
from Lemma \ref{prop_big_cluster}(i) we conclude that $P$ cannot
entirely lie in $C.$ On the other hand, we can assume, without loss
of generality,  that $C$ has a son $C'$ such that  $P\cap
C'\ne\emptyset$ and $P$ is entirely included in the $CC'$-fiber of
$G.$  Therefore, in each of the cells $C_1$ and $C_2$ one can pick a
vertex which is adjacent to a vertex of $C'.$ Let $a_1b_1$ and
$a_2b_2$ be two edges of $G$ such that $a_1\in C_1, a_2\in C_2,$ and
$b_1,b_2\in C'.$ Since, by Lemma \ref{prop_big_cluster},
$d_G(a_1,a_2)>\Delta+2=8\lambda+8,$ we conclude that
$d_G(b_1,b_2)\ge 8\lambda+6>4\lambda+2=\Lambda,$ thus $C'$ is
bifocal and the vertices $b_1$ and $b_2$ belong to different cells
$C'_1,C'_2$ of $C',$  say $b_1\in C'_1$ and $b_2\in C'_2.$ Suppose
now that $G$ contains an edge $uv$ with $u\in C_1$ and $v\in C'_2.$
Then, by Lemma \ref{prop_big_cluster}(i), we conclude that
$8\lambda+8=\Delta+2<d_G(u,a_2)\le 1+d_G(v,b_2)+1,$ whence
$d_G(v,b_2)>8\lambda+4=2\Lambda.$ Since $v,b_2\in C'_2,$ we obtain a
contradiction with Lemma \ref{bifocal_cluster}. Therefore,
contracting each of the cells $C_1,C_2,C'_1,C'_2$ into a vertex, we
will indeed obtain a $2K_2.$
\end{proof}

\begin{lemma} \label{big_cluster2} If a cluster $C'$ of a layering partition
${\mathcal LP}$ of $G$ is big or almost big, then its father $C$ is bifocal and
the neighbors in $C$ of the centers $c'_1$ and $c'_2$ of the cells
$C'_1$ and $C'_2$ of $C'$ belong to different cells of $C.$ In particular, any big or
almost big cluster is spread.
\end{lemma}

\begin{proof} Let $z_1$ and $z_2$ be two neighbors of $c'_1$ and $c'_2,$
respectively,   in $C.$ If $C$ is not bifocal, then $d_G(z_1,z_2)\le
4\lambda+2,$  whence $d_G(c'_1,c'_2)\le 4\lambda+4<16\lambda+10,$
contrary to the assumption that $C'$ is big or almost big. Thus, $C$
is bifocal. If $z_1$ and $z_2$ belong to the same cell $C_1$ of $C,$
then $d_G(z_1,z_2)\le 2\Lambda$, by Lemma \ref{bifocal_cluster}, and
therefore $d_G(c'_1,c'_2)\le 2\Lambda+2<16\lambda+10,$ leading to
the same contradiction as before.
\end{proof}

\begin{lemma} \label{big_cluster3} If a cluster $C$ of a layering partition ${\mathcal LP}$
of $G$ is big, then no son $C'$ of $C$ has a cell adjacent to both cells of $C.$
In particular, no big cluster $C$ has a small son adjacent to both cells of $C$.
\end{lemma}

\begin{proof} Let $C_1,C_2$ be the cells of $C.$ Suppose, by
way of contradiction, that two vertices $x',y'$ from the same  cell
of $C'$ are adjacent to vertices $x\in C_1$ and $y\in C_2,$
respectively. Then, by Lemma \ref{bifocal_cluster},  $d_G(x,y)\le
1+d_G(x',y')+1\le 8\lambda+4+2=8\lambda+6<\Delta+2,$  contrary to
the fact that, according to Lemma \ref{prop_big_cluster}, the
cluster $C$ is $(\Delta+2)$-separated.
\end{proof}

\subsection{The algorithm}
We continue with the description of an algorithm which, for an input
graph $G$ and a current value of ``the optimal distortion''
$\lambda,$ either establishes that no embedding with distortion $\le
\lambda$ of $G$ into an outerplanar metric exists or returns such an
embedding but with distortion at most $100\lambda+75$. Namely, given
a  value of $\lambda$ such that all clusters  of a layering
partition ${\mathcal LP}$ of $G$ contain at most two
$(4\lambda+2)$-far vertices,  if some cluster   of ${\mathcal LP}$
has two big sons or if this cluster is big and has two spread  sons,
then any embedding of $G$ in a $K_{2,3}$-minor free graph requires
distortion $>\lambda,$ and the algorithm returns the answer ``not".
Otherwise, if each cluster has at most one big son and each big
cluster has at most one spread son, then the algorithm constructs an
outerplanar graph $G'=(V,E').$ Then setting $w:=20\lambda+15$ as the
length of each edge of $G',$ the inequality $d_G(x,y)\le
d_{G'}(x,y)\le 5wd_G(x,y)$ holds for any two vertices $x,y$ of $V.$
To construct $G',$ the algorithm proceeds the clusters of ${\mathcal
LP}$ level by level in increasing order. To ensure that the
resulting graph $G'$ is outerplanar and the distortion of the
embedding of $G$ into $G'$ is bounded, we need to precise how the
algorithm ``opens" and ``closes" the cycles of $G'$, without allowing
cycles to  ``branch" and without incurring larger and larger
distortion. Roughly speaking, small and medium clusters of
${\mathcal LP}$ are used only to open or close cycles of $G'$ or to
build tree-components of $G'.$ Big clusters of ${\mathcal LP}$
are used to build-up the cycles of $G':$ each cycle $C$ of $G'$
starts and ends with vertices lying in small or medium clusters, all
other vertices of $C$ are pairs of centers of cells of big clusters
all lying in the same fiber. The remaining vertices of each cell of
a big cluster are made adjacent in $G'$ to the neighbor in $C$ of
the center of this cell; for an illustration, see Fig.
\ref{fig_algo}. Note that not every outerplanar graph can
occur as $G'$ returned by the algorithm because the cycles of $G'$
all have even length and  the 2-connected components of $G'$ are
edges or cycles. Moreover, no two cycles of $G'$ have a common origin.
The precise local  rules of constructing $G'$ are
provided in lines 3-7 of the algorithm described below.

\begin{figure}[t]
\begin{center}
\scalebox{0.5}
{\input{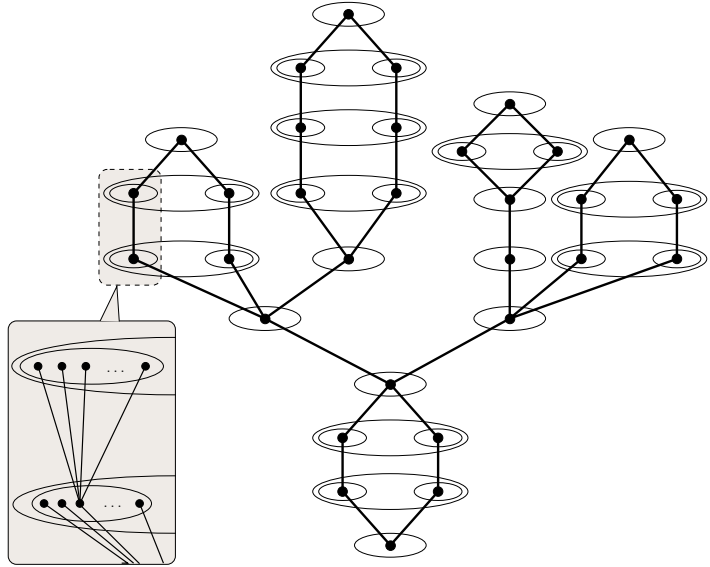}}
\end{center}
\caption{An outerplanar graph produced by the algorithm  {\sc Approximation by Outerplanar Metric}}
\label{fig_algo}
\end{figure}

\bigskip
\begin{center}
\framebox{
\parbox{14cm}{
\vspace{0.05cm}
\noindent{\bf Algorithm} {\sc Approximation by Outerplanar Metric}\\
{\footnotesize
  {\bf Input:} A graph $G=(V,E),$ a layering partition  ${\mathcal LP}$ of $G,$  and a value $\lambda$\\
  {\bf Output:} An outerplanar graph $G'=(V,E')$ or an answer ``not"\\
  \begin{tabular}[t!]{l@{ }p{14cm}}
  1. & {\bf For} each cluster $C$ of the layering partition ${\mathcal LP}$ {\bf do}\\
  3. & \begin{tabular}[t]{@{}p{3mm}@{}p{14cm}}& {\bf If} $C$ has two big sons or  $C$ is big and has two spread sons, {\bf then return} the answer ``not".\end{tabular}\\
  3. & \begin{tabular}[t]{@{}p{8mm}@{}p{11cm}} & {\bf Else for} each son $C'$ of  $C$ {\bf do}\end{tabular}\\
  4. & \begin{tabular}[t]{@{}p{13mm}@{}p{11cm}} &  {\sf Case 1:} {\bf
  If} $C'$ is small, {\bf then} pick in $C$ the center $c$ of a cell
  of $C$ adjacent to $C'$ and in $G'$ make   $c$ adjacent to all vertices of $C'.$\end{tabular}\\
  5. & \begin{tabular}[t]{@{}p{13mm}@{}p{11cm}} & {\sf Case 2:} {\bf
  If} $C'$ is medium and $C$ is not big, or $C'$ is medium and not
  spread and $C$ is big, {\bf then} pick in $C$ the center $c$ of a
  cell of $C$ adjacent to $C'$ and in $G'$ make    $c$ adjacent to all vertices of $C'.$\end{tabular}\\
  6. & \begin{tabular}[t]{@{}p{13mm}@{}p{11cm}} & {\sf Case 3:} {\bf
  If} $C'$ is medium, $C$ is  big, and $C'$ is the (unique) spread son
  of $C,$ {\bf then} in $G'$ make   the center $c_1$ of cell $C_1$ of
  $C$ adjacent to all vertices of $C'.$    Additionally, make  the
  center $c_2$ of  cell $C_2$ of $C$ adjacent to every vertex of $C'.$ \end{tabular}\\
  7. & \begin{tabular}[t]{@{}p{13mm}@{}p{11cm}} & {\sf Case 4:} {\bf
  If} $C'$ is big with  cells $C'_1,C'_2,$ such that $C'_1$ is
  adjacent to $C_1$ and $C'_2$ is adjacent to $C_2,$ where $C_1$ and
  $C_2$ are the cells of $C$ with centers $c_1$ and $c_2,$ {\bf then}
  in $G'$  make   $c_1$ adjacent to all vertices of $C'_1$ and   $c_2$
  adjacent to all vertices of $C'_2.$  \end{tabular}\\
  \end{tabular}}
\vspace{-0.3cm}\\ }}\hspace{0.2cm}
\end{center}
\medskip

\subsection{Correctness of the algorithm}

Now, we formulate the main results establishing the correctness and the approximation ratio of our algorithm.
The proofs will be provided in remaining subsections of this section.

\begin{theorem} \label{outerplanar} Let $G=(V,E)$ be an input graph and let $\lambda\ge 1.$ If the algorithm {\sc Approximation
by Outerplanar Metric} returns the answer ``not'', then any
embedding of $G$ into a $K_{2,3}$-minor free graph requires
distortion $>\lambda.$ Otherwise, if the algorithm returns the
outerplanar graph $G'=(V,E'),$ then uniformly assigning to its edges
weight $w:=20\lambda+15,$ we obtain an embedding of $G$ to $G'$ such
that $d_G(x,y)\le d_{G'}(x,y)\le 5wd_G(x,y)$ for any two vertices
$x,y$ of $V.$ As a result, we obtain a factor $100\lambda+75$
approximation of the optimal distortion of embedding a graph
distance into an outerplanar metric.
\end{theorem}

The proof of this theorem is subdivided into two propositions.
We start with a technical result, essentially showing that in both cases when our algorithm returns
the answer ``not", any embedding of $G$ into an outerplanar metric
requires distortion $>\lambda$:

\begin{proposition} \label{two_fibers} Let $C$ be  a big or an almost big cluster  having two sons $C',C''$ such
that  the two cells of $C$ can be connected in both $CC'$- and
$CC''$-fibers of $G.$ Then,  any embedding of $G$ in a
$K_{2,3}$-minor free graph requires distortion $>\lambda.$ These conditions are fulfilled
in the following two cases: (i) the cluster $C$ is big and has two spread sons; (ii) $C$ has two big sons $C',C''.$
\noindent
In particular, if the algorithm  returns the answer ``not", then any embedding of $G$ in  a $K_{2,3}$-minor free graph requires
distortion $>\lambda.$
\end{proposition}

Now suppose that the algorithm returns the graph $G'.$   We continue
 with the basic property of the graph $G'$ allowing us to analyze the
 approximation ratio of the
algorithm. First notice that, by construction, $G'$ is outerplanar.
Denote by  $d_{G'}(x,y)$ the distance in $G'$ between two vertices
$x$ and $y,$ where each edge of $G'$ has length $w:=20\lambda+15.$

\begin{proposition} \label{property_edge_G} For each edge $xy$ of the graph $G,$ the vertices $x$ and $y$ can be connected in the graph $G'$
by a path consisting of at most 5 edges, i.e. $d_{G'}(x,y)\le 5w$. Conversely, for each edge $xy$ of the graph $G',$ we have $d_G(x,y)\le 20\lambda+15.$
\end{proposition}

\subsection{Proof of Proposition \ref{property_edge_G}}
We start with first assertion. First
suppose that the edge $xy$ of $G$ is horizontal, i.e.,
$d_G(s,x)=d_G(s,y).$ Let $C$ be the cluster of $G$ containing this
edge. Then,  either $C$ is not  big or $C$ is big and $x,y$ belong
to the same cell of $C.$ In both cases, by construction of $G'$, we
deduce that $x$ and $y$ will be adjacent in $G'$ to the same vertex
from the father $C_0$ of $C,$ implying $d_{G'}(x,y)=2w.$ Now suppose that $xy$ is
vertical, say $x\in C, y\in C'$
and $C'$ is a son of $C.$ Denote by $C_0$ the father of $C.$ Let $z$
be a vertex of $C$ to which $y$ is adjacent in $G'.$ If $C$ is
small, medium, or $C$ is big but  $x$ and $z$ belong to the same
cell, then in $G'$ the vertices $z$ and $x$ will be adjacent to the
same vertex $x_{C_0}$ of the father $C_0$ of $C$, yielding
$d_{G'}(x,y)\le 3w.$  So, suppose that $C$ is big and the vertices
$z$ and $x$ belong to different cells $C_1$ and $C_2$ of $C$, say
$z\in C_1$ and $x\in C_2$. By Lemma \ref{big_cluster3}, the cluster
$C'$ is not small. According to the algorithm, $z$ is the center of
the cell $C_1,$ i.e., $z=c_1.$ Note also that $x$ and the center
$c_2$ of its cell are both adjacent in $G'$ to a vertex $x_{C_0}\in
C_0,$ whence $d_{G'}(x,c_2)=2w.$ If $C'$ is big and say $y\in C'_1,$
then since $y$ is adjacent to $z$ in $G',$ from the algorithm we
conclude that a vertex of $C'_1$ is adjacent in $G$ to a vertex of
$C_1.$ On the other hand, $y\in C'_1$ is adjacent in $G$ to $x\in
C_2.$ As a consequence, the cell $C'_1$ is adjacent in $G$ to both
cells $C_1$ and $C_2$ of $C,$ which is impossible by Lemma
\ref{big_cluster3}. So, the cluster $C'$ must be medium.  If $C$ has
a big son $C''$, then since both cells of $C$ are adjacent in $G$ to
the medium son $C',$ we obtain a contradiction with Proposition
\ref{two_fibers}(i). Hence, $C$ cannot have big sons. Moreover, by
Proposition \ref{two_fibers}, $C'$
is the unique spread son of $C$. According to the algorithm (see
{\sf Case 3}), the centers $z=c_1$ and $c_2$ of the cells of $C$ are
adjacent in $G'$ to a common vertex $u$ from $C'$, yielding $d_{G'}(z,c_2)=2w.$ As a
result, we obtain a path
with at most 5 edges connecting the vertices $y$ and $x$ in $G:$
$(y,z=c_1,u,c_2,x_{C_0},x).$  This concludes the proof of the first assertion of
Proposition \ref{property_edge_G}.

We continue with second assertion.
Any edge $xy$ of $G'$ runs between two
clusters lying in consecutive layers of $G$ (and $G'$); let
$x\in C$ and $y\in C',$ where $C$ is the father of $C'.$ In $G$, the
vertex $y$ has a neighbor  $x'\in C.$ Let $x'\ne x,$
otherwise there is nothing to prove. If $C$ is not big, then
$d_G(x,x')\le 16\lambda+12,$ whence $d_G(x,y)\le 16\lambda+13,$ and
we are done. So, suppose that the cluster $C$ is big.  If $x$ and
$x'$ belong to the same cell of $C,$ then Lemma
\ref{bifocal_cluster} implies that $d_G(x,x')\le
2\Lambda=8\lambda+4,$ yielding $d_G(x,y)\le 8\lambda+5.$ Now,
suppose that $x\in C_1$ and $x'\in C_2.$ By Lemma
\ref{big_cluster3}, $C'$ is a medium or a big cluster. If $C'$ is
big and $y\in C'_1$, since $x$ and $y$ are adjacent in $G',$
according to the algorithm, $C'_1$ contains a vertex that is
adjacent in $G$ to a vertex of $C_1.$ Since $y\in C'_1$ is adjacent
in $G$ to $x'\in C_2,$ we obtain a contradiction with Lemma
\ref{big_cluster3}. Hence $C'$ is a medium cluster. According to the
algorithm, $x$ is the center of the cell $C_1$ and $C_1$ contains a
vertex $z$ adjacent in $G$ to a vertex $v\in C'.$ Since $x,z\in C_1$
implies $d_G(x,z)\le 4\lambda+2$  and $y,v\in C'$ implies
$d_G(y,v)\le 16\lambda+12$, we obtain $d_G(x,y)\le 20\lambda+15.$

\subsection{Proof of Proposition \ref{two_fibers}}\label{sec:proofs-5-6}
By Proposition
\ref{metric-relaxed-minor},  it suffices to show that $G$ contains
$K_{2,3}$ as a  $\lambda$-metric relaxed minor.  Indeed,
suppose that $C$ is a big or an
almost big cluster with cells $C_1$ and $C_2$ having two sons $C',
C'',$ such that $C_1$ and $C_2$ can be connected by a path in each
of  the $CC'$- and $CC''$-fibers of $G.$ Let $k=d_G(s,C).$ Denote by
$P'$ and $P''$ the shortest two such paths connecting two vertices
of $C,$ one in $C_1$ and another in $C_2,$ in ${\mathcal F}(C,C')$
and ${\mathcal F}(C,C''),$ respectively. Denote by $x'\in C_1$ and
$y'\in C_2$ the end-vertices of $P'$ and by $x''\in C_1$ and $y''\in
C_2$ the end-vertices of $P''.$ Clearly, the choice of $P'$ implies
 $P'\cap C=\{ x',y'\}$ and the choice of $P''$ implies $P''\cap C=\{ x'',y''\}.$ Let $w'$
and $w''$ be middle vertices of $P'$ and $P'',$ respectively (if one
of these paths has odd length, then it has two middle vertices, and
we pick one of them). Let $a'$ and $b'$ be the vertices of $P'$
located at distance $\lambda+1$ (measured in $P'$) from $w',$ where
$a'$ is located between $w'$ and $x'$ and $b'$ is located between
$w'$ and $y'.$ Denote by $L'$ the subpath of $P'$ comprised between
$a'$ and $w'$ and by $R'$ the subpath of $P'$ comprised between $w'$
and $b'.$ Analogously, for $P''$ we can define the vertices
$a'',b''$ and the paths $L''$ and $R''$ of length $\lambda+1$ each.
Finally, denote by $P'_1$ and $P'_2$ the subpaths of $P'$ comprised
between $a'$ and $x'$ and between $b'$ and $y'.$ Analogously, define
the supbaths $P''_1$ and $P''_2$ of $P''.$  Pick any shortest path
$M'$ in $G$ between the vertices $x',x''$ and any shortest path
$M''$ between $y',y''.$ Let $F'$ be a subpath of a shortest path
$P(x',s)$ from $x'$ to the root $s$ starting with $x'$ and having
length $3\lambda.$  Analogously, let $F''$ be a subpath of a
shortest path $P(y'',s)$ from $y''$ to  $s$ starting with $y''$ and
having length $3\lambda.$ Let $J'$ and $J''$ be the subpaths of
length $\lambda+1$ of $P(x',s)$ and $P(y',s),$ which continue $F'$
and $F''$, respectively, towards $s;$ see Fig. \ref{fig:two_fibers}
for an illustration.

\begin{figure}[t]
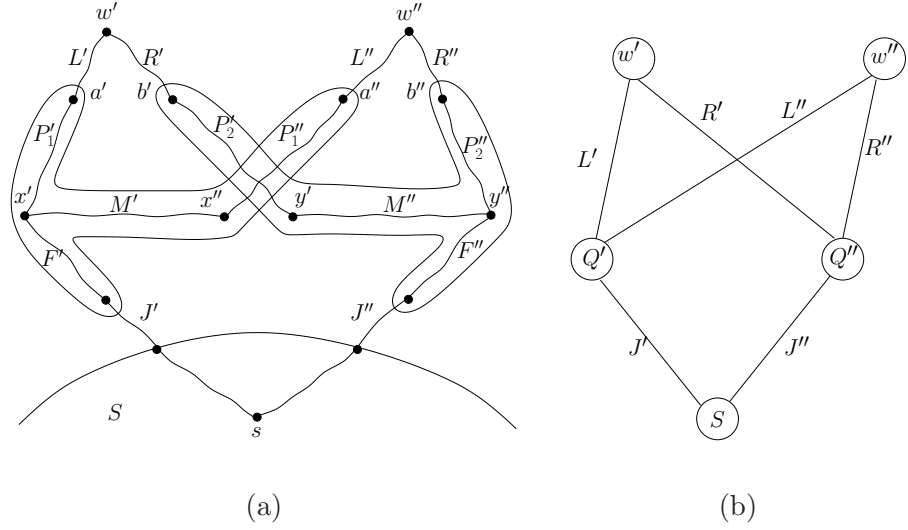

\begin{center}
\begin{tabular}{ccc}
\scalebox{0.35}
{\input{outer.pstex_t}} & &
\scalebox{0.35}
{\input{two_fibers.pstex_t}}\\
~ & & \\
(a) & ~~~~~~~~~~~~~~~~~~~~~~ & (b)
\end{tabular}
\end{center}
\caption{The {\it elements} of the map $\mu$.}
\label{fig:two_fibers}
\end{figure}

Now we are ready to define a mapping  $\mu: V(K_{2,3})\cup
E(K_{2,3})\mapsto V(G)$ certifying that $K_{2,3}$ is a
$\lambda$-metric relaxed minor of $G.$ Denote the vertices of
$K_{2,3}$ by $a,b,c,q',q'',$ where the vertices $q'$ and $q''$ are
assumed to be adjacent to each of the vertices $a,b,c.$ We set
$\mu(a):=\{ w'\}, \mu(b):=\{ w''\}, \mu(q'):=P'_1\cup P''_1\cup
M'\cup F'=:Q', \mu(q''):=P'_2\cup P''_2\cup M''\cup F'':=Q'',$ and
$\mu(c):=B_{k'}=:S,$ where $k'=k-4\lambda-1.$ Additionally, for each
edge of $K_{2,3},$ we set $\mu(aq'):=L',\mu(aq''):=R',
\mu(bq'):=L'',\mu(bq''):=R'', \mu(q's):=J', \mu(q''s):=J''.$ We will
call the paths
$L',L'',R',R'',P'_1,P'_2,P''_1,P''_2,F',F'',J',J'',M',M'',$ the
vertices $w',w'',$ and the set $S$ the {\it elements} of the map
$\mu$. Notice first that  each vertex of $K_{2,3}$ is mapped to a
connected subgraph of $G$ and each edge of $K_{2,3}$ is mapped to a
path of $G,$ thus $\mu$ satisfies the conditions (i) and (iii) of a
metric relaxed minor. It remains to show that $\mu$ satisfies the
conditions (ii$^+$), (iv$^+$), and (v$^+$). The proof of this is
subdivided into several intermediate  results.

\begin{lemma} \label{paths_P'P''} $L'\cup R'$ and $L''\cup R''$ are shortest paths of $G.$
\end{lemma}

\begin{proof} Suppose, by way of contradiction, that the vertices $a'$ and $b'$ can be
connected in $G$ by a path $P_0$ shorter than $L'\cup R',$ in
particular, $d_G(a',b')\leq 2\lambda+1.$  Since the length of $P'$
is greater than $\Delta,$ we conclude that the vertices $a'$ and
$b'$ do not belong to the cluster $C.$  From the choice of $P',$ the
path $P_0$ necessarily  contains  vertices of $B_{k-1}(s),$ and
therefore $P_0$ necessarily traverses the cluster $C.$ First,
suppose that $P_0$ intersects only one cell of $C,$ say $C_1.$ Let
$u$ be the last vertex of $C_1\cap P_0,$ while moving from $a'$ to
$b'$ along $P_0.$ The length of the subpath $Q_0$ of $P_0$ comprised
between $u$ and $b'$ is strictly less than the length of $P_0$ and
therefore than the length of $L'\cup R'.$   Since $b'$ belongs to
the fiber ${\mathcal F}(C,C')$ but does not belong to $C,$ we
conclude that necessarily $Q_0$ is contained in ${\mathcal
F}(C,C').$  As a result, the vertices $u$ and $y'$ can be connected
in the fiber ${\mathcal F}(C,C')$  by a path $Q_0\cup P'_2$ shorter
than $P',$ contrary to the choice of $P'.$ Now, suppose that the
path $P_0$ intersects both cells of $C.$ Pick $u\in P_0\cap C_1$ and
$v\in P_0\cap C_2.$ Since $u$ and $v$ can be connected in $G$ by the
subpath of $P_0$ comprised between them, we conclude that
$d_G(u,v)\le 2\lambda+1,$ contrary to the assumption that $C$ is a
big or almost big cluster, and thus a $\Delta$-separated
cluster.
\end{proof}

\begin{lemma} \label{distance_P'P''_to_C}
If $z\in L'\cup R',$  then $d_G(z,C)\ge 4\lambda+3-d_G(z,w')\ge
3\lambda+2.$ Analogously, if $z\in L''\cup R'',$ then $d_G(z,C)\ge
4\lambda+3-d_G(z,w'')\ge 3\lambda+2.$ In particular, $d_G(w',C)\ge
4\lambda+3$ and $d_G(w'',C)\ge 4\lambda+3.$
\end{lemma}

\begin{proof} Note that the length of the subpath of $P'$ between $x'$ and
$w'$ is at least $4\lambda+3$ as $C$ is ($8\lambda+6$)-separated.
Consequently, the length of the subpath of $P'$ between $x'$ and
$z\in L'$ is at least $4\lambda+3-d_G(z,w')$. Assume, by way of
contradiction, that $d_G(z,C)<4\lambda+3-d_G(z,w').$ First suppose
that $d_G(z,C)=d_G(z,u)$ for $u\in C_1.$ Any shortest $(z,u)$-path
$P(u,z)$ lies entirely in the fiber ${\mathcal F}(C,C')$ (and
therefore outside the ball $B_{k-1}(s)$).  Since
$d_G(z,u)<4\lambda+3-d_G(z,w')$,
$d_G(z,u)$ is less than the length of the subpath of $P'$ between
$x'$ and $z$.   Hence, we conclude that the $(u,y')$-path consisting
of the path $P(u,z)$ followed by the subpath of $P'$ between $z$ and
$y'$ is contained in ${\mathcal F}(C,C')$ and is shorter than $P',$
contrary to the choice of $P'.$ Now suppose that $d_G(z,C)=d_G(z,u)$
for a vertex $u\in C_2.$ Since $z\in L'$ and $u\in C$ both belong to
${\mathcal F}(C,C'),$ any shortest $(z,u)$-path $P(u,z)$ also
belongs to this fiber. Note that $P(u,z)$ has length $<4\lambda+3$
while the subpath of $P'$ between $z$ and $y'$ has length $\ge
4\lambda+3.$ Therefore, the path between $u$ and $x'$ consisting of
$P(u,z)$ followed by the subpath of $P'$ between $z$ and $x'$ is
shorter than $P'$ and is contained in the fiber  ${\mathcal
F}(C,C'),$  contrary to the choice of $P'.$
\end{proof}

\begin{lemma} \label{setS} The set $S$ is $\lambda$-far from all elements of $\mu$ except $J'$,
$J''$ and itself.
\end{lemma}

\begin{proof} From the definition of the layering it follows that $S$ is $\lambda$-far
from the paths $F',F'',$ and the cluster $C.$ Since any path from a
vertex of $S$  to the $CC'$- and $CC''$-fibers traverses the cluster
$C,$ we conclude that $S$ is $\lambda$-far from the vertices
$w',w''$ and the paths $L',R',L'',R'',P'_1,P'_2,P''_1,P''_2.$ It
remains to show that $S$ is $\lambda$-far from the paths $M'$ and
$M''.$ Suppose, by way of contradiction, that $d_G(u,v)\le \lambda$
for $u\in M'$ and $v\in S.$ Let $d_G(u,x')\le d_G(u,x'').$ Since the
length of $M'$ is at most $4\lambda+2$ (by Lemma
\ref{prop_big_cluster}(ii)), we conclude that $d_G(u,x')\le
2\lambda+1,$ whence $d_G(x',v)\le 3\lambda+1,$ contrary to the
assumption that $d_G(s,x')=k$ and $d_G(s,v)\le k'=k-4\lambda-1.$
\end{proof}

\begin{lemma} \label{setw'w''} The vertex $w'$ is $\lambda$-far from all elements of $\mu$ except $L'$, $R'$  and itself.
Analogously, $w''$ is $\lambda$-far from all elements of $\mu$
except $L''$, $R''$  and itself.
\end{lemma}

\begin{proof} From Lemma \ref{distance_P'P''_to_C} we conclude that
$d_G(w',C)\ge 4\lambda+3.$ Since any path between $w'$ and a vertex
of the set $\{ w''\}\cup L''\cup P''_1\cup R''\cup P''_2\cup F'\cup
F''\cup J'\cup J''\cup S$ traverses the cluster $C,$ we conclude
that $w'$ is $\lambda$-far from each of these elements of $\mu$.
Next we show that $w'$ is $\lambda$-far from the paths $M'$ and
$M''.$ Suppose, by way of contradiction, that $d_G(w',u)\le \lambda$
for a vertex $u\in M'$ and assume, without loss of generality, that
$u$ is closer to $x'$ than to $x'',$ yielding $d_G(x',u)\le
2\lambda+1.$ But then $d_G(w',x')\le 3\lambda+1<4\lambda+3,$
contrary to the assumption that $d_G(w',C)\ge 4\lambda+3.$ Finally,
we will show that $w'$ is $\lambda$-far from the paths $P'_1$ and
$P'_2.$ Suppose, by way of contradiction,  that $d_G(w',u)\le
\lambda$ for a vertex $u\in P'_1.$ Let $P(u,w')$ be a shortest
$(u,w')$-path. Obviously, $P(u,w')$ belongs to the $CC'$-fiber
${\mathcal F}(C,C').$  Hence, replacing in $P'$ the subpath
comprised between $u$ and $w'$ (and comprising $L'$) by $P(u,w')$,
we obtain a shorter path connecting $x'$ and $y'$ in ${\mathcal
F}(C,C').$  This contradiction shows that $w'$ is $\lambda$-far from
$P'_1$ and $P'_2.$
\end{proof}

\begin{lemma} \label{setL'R'} Each of the paths $L'$ and $R'$ is
$\lambda$-far from each of the elements $L'',R'',P''_1,P''_2,J',J''$
of $\mu.$ Analogously, $L''$ and $R''$ are $\lambda$-far from
$P'_1,P'_2,J',J''.$ In particular, the $\mu$-images of any two
non-incident edges of $K_{2,3}$ are $\lambda$-far.
\end{lemma}

\begin{proof}
The cluster $C$ separates the $CC'$-fiber containing  $L'\cup R'$
from the rest of the graph. Therefore, any path connecting a vertex
$u\in L'\cup R'$ to a vertex $v\in L''\cup R''\cup P''_1\cup
P''_2\cup J'\cup J''$ traverses $C.$ Since $d_G(u,C)\ge
4\lambda+3-d_G(u,w')\ge 4\lambda+3-\lambda-1=3\lambda+2>\lambda,$ we
conclude that $d_G(u,v)>\lambda.$
\end{proof}

\begin{lemma} \label{setQ'_Q''} The set $Q'$ is $\lambda$-far from the
paths $R',R''$ and $J''.$ Analogously, the set $Q''$ is $\lambda$-far
from the paths $L',L'',$ and $J'.$
\end{lemma}

\begin{proof} That $J''$ is $\lambda$-far from $P'_1$ and $P''_1$ follows
from the definition of $J''$ and the fact that the cluster $C$
separates  $J''$ from $P'_1\subseteq {\mathcal F}(C,C')$ and
$P''_1\subseteq {\mathcal F}(C,C'')$. Now suppose that
$d_G(u,v)\le\lambda$ for $u\in M'\cup F'$ and $v\in J''\setminus S.$
If $u\in F',$ then $d_G(x',u)\le 3\lambda, d_G(v,y'')\le 4\lambda$
and, by triangle inequality, we conclude that $d_G(x',y'')\le
8\lambda<\Delta,$ contrary to the fact that $x'\in C_1,y''\in C_2$ and the
cluster $C$ is $\Delta$-separated. If $u\in M',$ then since
$M'$ has length at most $4\lambda+2,$ we conclude that one of the
vertices $x',x'',$ say $x'$ has distance at most $2\lambda+1$ to
$u.$ Then,  by triangle inequality, again we conclude that
$d_G(x',y'')\le 7\lambda+1,$ contrary with
$\Delta$-separability of $C.$ This shows that $Q'$ and $J''$
are $\lambda$-far.

It remains to show that  $Q'$ and $R'\cup R''$ are $\lambda$-far.
Pick $u\in Q'$ and $v\in R''.$ By  Lemma \ref{distance_P'P''_to_C},
any vertex $v\in R''$ is located at distance $\ge 3\lambda+2$ from
the cluster $C.$ Since $C$ separates $R''$ from $P'_1$ and $F',$ we
conclude that $d_G(u,v)\ge 3\lambda+2$ for any vertex $u\in P'_1\cup
F'.$ If $u\in M'$ and $d_G(x',u)\le d_G(x'',u),$ then $d_G(x',u)\le
2\lambda+1,$ yielding $d_G(x',v)\le d_G(x',u)+d_G(u,v)\le
2\lambda+1+d_G(u,v).$ Hence, if $d_G(u,v)\le \lambda$, we get
$d_G(x',v)\le 3\lambda+1,$ contrary to the fact that $d_G(v,C)\ge
3\lambda+2.$ Finally, suppose that $u\in P''_1$ and $d_G(u,v)\le
\lambda.$ Let $P_0$ be any shortest path between $u$ and $v.$
Replacing the subpath $P''(u,v)$ of $P''$ comprised between $u$ and
$v$ by $P_0$, we will obtain a path $P$ shorter than $P''$ (because
$L''\subset P''(u,v)$ and, by Lemma \ref{paths_P'P''}, $L''$ is a
shortest path of length $\lambda+1$ of $G$). The path $P$ is
completely contained in the union of the fiber ${\mathcal F}(C,C'')$
and the ball $B_k(s).$ Moreover, each time $P$ moves from ${\mathcal
F}(C,C'')$ to $B_k(s),$ it traverses the cluster $C.$ Therefore,
taking any subpath of $P$ between two vertices from different cells
of $C$ and completely contained in ${\mathcal F}(C,C''),$ we will
obtain a contradiction with the minimality choice of the path $P''.$
This contradiction concludes the proof that $Q'$ is $\lambda$-far
from $R',R''$ and $J''.$
\end{proof}

\begin{lemma} \label{setsQ'Q''} The sets $Q'$ and $Q''$ are $\lambda$-far.
\end{lemma}

\begin{proof} First notice that $M'$ and $M''$ are $\lambda$-far.
Indeed,  pick $u\in M', v\in M'',$ and suppose, without loss of
generality,   that $d_G(x',u)\le d_G(x'',u)$ and $d_G(y'',v)\le
d_G(y',v).$ If $d_G(u,v)\le \lambda,$ then, by triangle inequality,
$d_G(x',y'')\le d_G(x',u)+d_G(u,v)+d_G(v,y'')\le
2\lambda+1+\lambda+2\lambda+1=5\lambda+2<\Delta,$ contrary to
assumption that $C$ is $\Delta$-separated. In a similar way
one can show that $F'$ and $M''$ as well as $M'$ and $F''$ are
$\lambda$-far: if $d_G(u,v)\le\lambda$ for $u\in F'$ and $v\in M''$
with $d_G(y'',v)\le 2\lambda+1,$ then $d_G(x',y'')\le
3\lambda+\lambda+2\lambda+1=6\lambda+1<\Delta.$ Analogously, if
$F'$ and $F''$ are not $\lambda$-separated, then $d_G(x',y'')\le
3\lambda+\lambda+3\lambda=7\lambda<\Delta,$ a contradiction.

Suppose now that $u\in Q',v\in P''_2,$ and $d_G(u,v)\le \lambda.$
Let  $P_0$ be a shortest path of $G$ between $u$ and $v.$ Since $C$
is $\Delta$-separated, $P_0$ cannot intersect both cells $C_1$
and $C_2$ of $C.$ On the other hand, since $v\in P''_2\subset
{\mathcal F}(C,C''),$ the path $P_0$ necessarily contains a vertex
$v_0\in C$ such that the whole subpath of $P_0$ between $v_0$ and
$v$ is contained in ${\mathcal F}(C,C'').$ If $v_0\in C_1,$ then the
path constituted by the subpath of $P_0$ between $v_0$ and $v,$
followed by the subpath of $P''_2$ between $v_0$ and $y'',$ is
completely contained in the fiber ${\mathcal F}(C,C'')$ and is
shorter than $P''$ (because $L''\cup R''$ has length $2\lambda+2$),
contrary to the minimality choice of $P''.$  Therefore, necessarily
$v_0\in C_2,$ showing also that $P_0\cap C=P_0\cap C_2.$ Let also
$u_0$ be the first intersection of $P_0$ with $C_2$ while moving
from $u$ to $v.$

If $u\in F'$ then $d(x',u)\le 3\lambda$ and we conclude
$d_G(x',u_0)\le d_G(x',u)+d_G(u,u_0)\le
3\lambda+\lambda=4\lambda<\Delta,$ contrary to the fact that $C$
is $\Delta$-separated. Analogously, if $u\in M'$ and
$d_G(x',u)\le d_G(x'',u),$ then $d_G(x',u_0)\le
d_G(x',u)+d_G(u,u_0)\le 2\lambda+1+\lambda<\Delta.$  If $u\in
P'_1,$ then the subpath of $P'$ between $x'$ and $u$ followed by the
subpath of $P_0$ between $u$ and $u_0$ forms a path  contained in
the fiber ${\mathcal F}(C,C')$ and that is shorter that $P'$
(because $L'\cup R'$ has length $2\lambda+2$), contrary to the
choice of $P'.$ Finally, suppose that $u\in P''_1.$ Then, the
subpath of $P''$ between $x''$ and $u$ followed by the subpath of
$P_0$ between $u$ and $u_0$ constitute a path contained in the fiber
${\mathcal F}(C,C'')$ and is shorter that $P'',$ contrary to the
choice of $P''.$ This contradiction shows that the sets $Q'$ and
$Q''$ are $\lambda$-far.
\end{proof}

This establishes the first assertion of Proposition  \ref{two_fibers}.
To prove the second assertion of Proposition \ref{two_fibers}, first
suppose suppose
that the cluster $C$ is big and $C$ has a big and a medium sons
$C',C''$ such that both cells $C_1$ and $C_2$ are adjacent to $C''$
or that $C$ has two medium sons $C',C''$ adjacent to both cells of
$C.$  By definition of the layering, each vertex of $C'\cup C''$ is
adjacent to a vertex of $C.$ If all vertices of $C'$ are adjacent to
vertices from the same cell of $C,$ say $C_1,$ then for any
$x',y'\in C'$ we have $d_G(x',y')\le 2+4\lambda+2,$ contrary to the
assumption that $C'$ is big. Hence, both cells of $C$ are adjacent
to  $C',$ say $x\in C_1$ is adjacent to $x'\in C'$ and $y\in C_2$ is
adjacent to $y'\in C'.$ By Lemma \ref{big_cluster3}, $x'$ and $y'$
belong to different cells of $C',$ say $x'\in C'_1$ and $y'\in
C'_2.$ Let $k:=d_G(s,C).$ Since $x',y'\in C',$ the vertices $x'$ and
$y'$ are adjacent in $G(V\setminus B_k(s))$ by a path $P(x',y').$
Then $P(x,y):=xx'\cup P(x',y')\cup y'y$ is a path between $x$ and
$y$ in the $CC'$-fiber ${\mathcal F}(C,C').$ Analogously, since
both cells $C_1$ and $C_2$ are adjacent to $C'',$ we conclude that
two vertices from different cells of  $C$ can be connected by a path
belonging to the $CC''$-fiber, showing that the conditions of Proposition
\ref{two_fibers} are fulfilled. This concludes the proof establishes in case (i).

Now suppose that $C$ has two big sons $C'$ and $C''.$ Then $C$ is
either a big or an almost big cluster. By Lemma \ref{big_cluster1},
each of the clusters $C',C''$ is $(8\lambda+8)$-separated while the
cluster $C$ is $(8\lambda+6)$-separated and that its cells $C_1$ and $C_2$
have diameters at most $\Lambda.$ As in previous cases, one can
deduce that $C_1$ is adjacent to one cell of each of the clusters
$C'$ and $C'',$ while $C_2$ is adjacent to the second cell of these
clusters, establishing the case (ii) and concludung the proof of
Proposition  \ref{two_fibers}.

\subsection{Proof of Theorem \ref{outerplanar}} The algorithm returns the
answer ``not'' when a cluster $C$  has two big sons or a big cluster
$C$ has two spread sons. In this case, by Proposition
\ref{two_fibers} any embedding of $G$ into a $K_{2,3}$-minor free
graph requires distortion $>\lambda,$ whence $\lambda^*(G,{\mathcal
O})>\lambda.$ Now suppose that the algorithm returns the outerplanar
graph $G'$ weighted uniformly with $w=20\lambda+15.$ Notice that in Case 4 of the algorithm,
the required matching between the four cells of the big clusters $C$ and $C'$ exists by Lemma \ref{big_cluster1}
and because $C'$ is the unique spread son of $C.$  By Proposition
\ref{property_edge_G} we have $d_G(x,y)\le
20\lambda+15=d_{G'}(x,y)$ for each edge $xy$ of the graph $G'.$  By
Lemma \ref{distortion_edge2} we conclude that $d_G(x,y)\le
d_{G'}(x,y)$ for any pair $x,y\in V.$ By Proposition
\ref{property_edge_G}, for any edge $xy$ of $G,$ the vertices $x$ and
$y$ can be connected in $G'$ by a path with at most 5 edges, i.e.,
$d_{G'}(x,y)\le 5w=100\lambda+75.$ By Lemma  \ref{distortion_edge1}
we conclude that $d_{G'}(x,y)\le (100\lambda+75)d_G(x,y)$ for any
pair $x,y$ of $V.$ Hence $d_G\le d_{G'}\le (100\lambda+75)d_G,$
concluding the proof of Theorem \ref{outerplanar}.

\section{Final remarks} All our non-algorithmic results hold for infinite graphs as well.
Namely, given a connected graph $G$ with an arbitrary number of
vertices, we can define a layering partition ${\mathcal LP}$ of $G.$
Then, as in Sections 2 and 3, the largest diameter of a cluster of
${\mathcal LP}$ can be used to upper bound the optimal (additive or
multiplicative) distortion of embedding $G$ into a tree metric. The
trees $H,H',H_{\ell},$ and $H'_{\ell}$ can be defined as in the
algorithm {\sc Approximation by Tree Metric} and these trees have
the same  approximation qualities as the analogous trees defined in
the finite case (see Corollaries 4-7). Therefore, an infinite graph
$G$ admits an embedding into a tree-metric with a finite distortion
if and only if the diameters of clusters of an arbitrary layering
partition ${\mathcal LP}$ of $G$ are uniformly bounded. Similar
conclusions hold for approximation by outerplanar graphs:
Propositions 4, 5, and 6 establish in what cases the optimal
distortion is $>\lambda;$ otherwise, the construction provided by the
algorithm {\sc Approximation by Outerplanar Graph} returns an
infinite outerplanar graph $G'$  into which $G$ embeds with
distortion $\le 100\lambda+75.$ The proof of Proposition 2 (done by
induction on the number of vertices and edges of $G$) seems to be an
obstacle to this conclusion. However, if a graph $G$ has a {\it
finite} graph $H$ as a relaxed minor, then  one can easily find a
finite subgraph $G_0$ of $G$ which still has $H$ as a relaxed minor,
and therefore we can apply Proposition 2 to $G_0$ instead of $G$ to
conclude that $G_0$ (and therefore $G$) has $H$ as a minor.

We conclude with two open questions. Our Proposition \ref{K_2,r} presents a strong
necessary condition for embedding a graph $G$ into a $K_{2,r}$-minor free metric
with a distortion $\leq \lambda$.  However, we were not able to provide all
structural conditions and to design a constant factor approximation
algorithm for this problem. An even more challenging problem is
designing a constant factor approximation algorithm for optimal
distortion of embedding a graph metric into a $K_4$-minor free
metric (series-parallel metric).


\bigskip\noindent
{\bf Acknowledgment.}  This research
was partly supported by the ANR grant
BLANC ``GGAA:  Aspects G\'eom\'etriques, analytiques et algorithmiques des groupes".


\end{document}